\documentclass[EJP,preprint]{ejpecp} 
\usepackage{mathrsfs}
\usepackage{bbm}
\usepackage{bbold}
\usepackage{mathtools}
\usepackage{hyperref}

\usepackage[babel]{csquotes}
\usepackage{enumitem}
\usepackage{soul}

\SHORTTITLE{Limit theorems for multifractal products of random fields}

\TITLE{Limit theorems for multifractal products of random fields\support{This research was partially supported under the Australian Research Council's Discovery Projects funding scheme (project number  DP220101680).}}

\AUTHORS{%
  Illia~Donhauzer\footnote{La Trobe University, Australia. \EMAIL{I.Donhauzer@latrobe.edu.au}}
  \and 
  Andriy~Olenko\footnote{La Trobe University, Australia.
    \EMAIL{A.Olenko@latrobe.edu.au}}}


\AMSSUBJ{60F17; 60F25; 60G57; 60G60} 

\SUBMITTED{February 7, 2022} 
\ACCEPTED{-} 

\VOLUME{0}
\YEAR{2022}
\PAPERNUM{0}
\DOI{10.1214/YY-TN}

\ABSTRACT{ This paper investigates asymptotic properties of multifractal products of random fields. The obtained limit theorems provide sufficient conditions for the convergence of cumulative fields in the spaces $L_q.$ New results on the rate of convergence of cumulative fields are presented.  Simple unified conditions for the limit theorems and the calculation of the R{\'e}nyi function are given.  They are less restrictive than those in the known one-dimensional results. The developed methodology is also applied to  multidimensional multifractal measures. Finally, a new class of examples of geometric $\varphi$-sub-Gaussian random fields is presented. In this case, the general assumptions have a simple form and can be expressed in terms of covariance functions only.}

\DeclareMathOperator*{\esssup}{ess\,sup}

\begin{document}

\section{Introduction}

Multifractal temporal and spatial data have been observed in many applications, for example, environmental processes (precipitations fields), engineering (teletraffic), and cosmology \cite{leonenko2021analysis, Man, pathirana2002multifractal}. Jaffard \cite{jaffard1999multifractal} showed that the class of multifractal random processes is wide and all L{\'e}vy processes except Brownian motion and Poisson processes are multifractal. The idea of multifractals initially was proposed by Mandelbrot who pointed out that some systems might possess many scaling rules, contrary to the case of fractals, which can be described by a single fractal dimension. The R{\'e}nyi function is linked to the multifractal spectrum \cite{seuret2017multifractal} by the Legendre transform, and is an important tool in the analysis of multifractal processes. The R{\'e}nyi function depends more regularly on data than the multifractal spectrum, and often used in practice as it can easily be handled analytically and numerically \cite{riedi1995improved}.

There are several approaches to build mathematical models for multifractal random measures. The most popular ones are based on binomial cascades and branching measures on Galton-Watson trees (see \cite{reidi2002, kahane1987, molchan1996, shieh2002multifractal}). An approach to construct a multifractal random measure $\mu(\cdot)$ as a limit of random measures $\mu_m(\cdot)$ generated by cumulative processes $A_m(\cdot)$ defined as multifractal products of random processes was considered in numerous research studies. Kahane \cite{kahane1987} proved that the sequence of random measures $\mu_m(\cdot)$ converges weakly almost surely to a random measure $\mu(\cdot)$. Molchan \cite{molchan1996} studied Mandelbrot's random cascade measures and calculated the R{\'e}nyi function and multifractal dimensions for the case of a general generator. It was found that for these types of models, the R{\'e}nyi function can have discontinuities in its derivatives. Mannersalo et al. \cite{Man} studied the asymptotic behaviour of the cumulative processes $A_m(\cdot)$ and found necessary and sufficient conditions for a pointwise convergence of $A_m(\cdot),$ as $m\to\infty,$ in the space $L_2.$ Under some strict assumptions, the R{\'e}nyi function of the multifractal measure $\mu(\cdot)$ was computed on the interval $q\in [1,2].$ The important novelty of the results obtained in \cite{Man} is the stationarity of  the multifractal random measure $\mu(\cdot),$ which is not always the case for other models including cascading measures. Conditions providing continuity and nondegeneracy of the limit process $A(\cdot)$ were also obtained. Denisov and Leonenko \cite{denisov2016limit} studied conditions for the pointwise convergence of the cumulative processes $A_m(\cdot)$ in the spaces $L_q, \ q>0,$ and calculated the R{\'e}nyi function of the multifractal measure $\mu(\cdot)$ on the interval $[0,p], \ p> 0.$  Their conditions were significantly simpler compared to \cite{Man} and stated in terms of higher-order moments of the underlying random processes. The obtained results were also specified for some cases where the R{\'e}nyi function was calculated explicitly.

Unfortunately, there are only a few cases when the R{\'e}nyi function of the random measure $\mu(\cdot)$ was calculated explicitly. Anh and Leonenko \cite{anh2008log} derived the R{\'e}nyi function for the multifractal random measure $\mu(\cdot)$ under log-normal, log-gamma and log-negative inverted gamma scenarios. Anh et al. \cite{anh2008multifractality} constructed multifractal processes based on multifractal products of geometric Ornstein-Uhlenbeck processes driven by L{\'e}vy motion with inverse Gaussian or normal inverse Gaussian distribution. However, in a general case, conditions guaranteeing that the multifractal measure $\mu(\cdot)$ belongs to the spaces $L_q, \ q>0,$ and the explicit calculation of the R{\'e}nyi function of $\mu(\cdot)$ require further investigations. These problems are even more complex for the multidimensional case as random measures generated by multifractal products of random fields have been less frequently studied.

Multifractal measures and processes defined on multidimensional domains arise in numerous applications. For instance, Mandelbrot \cite{mandelbrot1989multifractal} showed examples of multifractal data in geophysics, while Pathirana and Herath discussed the multifractality of rainfields~\cite{pathirana2002multifractal}. Leonenko and Shieh \cite{leonenko2013renyi} studied the multifractal product of random fields on the sphere by generalizing the results obtained in \cite{Man}. Sufficient conditions for the pointwise convergence of cumulative fields in the space $L_2$ were obtained. In \cite{leonenko2021analysis}, Leonenko et al., motivated by the analysis of cosmic microwave background radiation data, studied multifractal random fields and measures defined on spheres. They developed specific models of multifractal fields where the R{\'e}nyi function can be computed explicitly. However, there are still numerous open problems, in particular about conditions providing the convergence in the spaces $L_q,$ for any $q>0,$ rates of convergence, and  explicitly calculating the R{\'e}nyi function for  new classes of processes.

The main focus of this investigation is to study the multifractal measure $\mu(\cdot)$ defined as a limit of measures $\mu_m(\cdot)$ generated by multifractal products of random fields. First, we generalize limit theorems obtained in \cite{denisov2016limit} by considering multifractal measures on the hypercube $[0,1]^n, \ n\in\mathbb{N}$. Compared to \cite{denisov2016limit}, a modified method is used along with more general mixing conditions on underlying homogeneous and isotropic random fields. The obtained results show the effect of the dimensionality and study several important scenarios.

The novelty of the paper compared to the existing literature is:
\begin{itemize}
\item[--] conditions on the pointwise convergence of the cumulative fields $A_m(\cdot)$ defined on the hypercube $[0,1]^n, \ n\in\mathbb{N},$ in spaces $L_q, \ q > 0,$
\item[--] simplified conditions for the calculation of the R{\'e}nyi function that are less restrictive than in the known one-dimensional results \cite{denisov2016limit, Man},
\item[--] rates of convergence in the spaces $L_q, \ q>0,$ and their analysis,
\item[--] unified assumptions that are stated in the same terms for all results,
\item[--] examples for the case of $\varphi$-sub-Gaussian random fields. To our best knowledge, only the geometric Gaussian and related scenarios have been studied in the existing literature \cite{leonenko2021analysis}. The class of $\varphi$-sub-Gaussian distributions includes compactly supported distributions, centred Weibull distributions, Gaussian, centred Poisson distributions, etc. Thus, the obtained results substantially enlarge the classes of processes for applications of multifractal methods.
\end{itemize}

To prove the main results reported in this paper, the martingale property of the sequence of random variables $A_m(\textbf{\textit{t}})$ is employed. Using this property, we obtain the sufficient conditions for the convergence and the rates of convergence in the spaces $L_q, \ q>0.$ It allows to estimate the moments of the limit measure $\mu(\cdot),$ which is used to calculate the R{\'e}nyi function.

The structure of the paper is as follows. Section \ref{sec1} provides the main definitions and notations that are used in the paper. Section \ref{sec2} gives the sufficient conditions for the pointwise convergence of cumulative fields in $L_q, \ q>0.$ Section \ref{rate_lq_sec} provides the rates of convergence for the obtained limit theorems. Section \ref{sec3} derives the R{\'e}nyi function of the random multifractal measure $\mu(\cdot)$ introduced in Section \ref{sec2}. Finally, Section \ref{sec4} investigates a special case of cumulative fields generated by $\varphi$-sub-Gaussian random fields.

\section{Main definitions and notations}
\label{sec1}
This section introduces main definitions and notations used in the paper.

In the following vectors will be  written in the bold face type (e.g. $\textit{\textbf{t}}$), while a regular font will be used to denote numbers and scalar variables. Throughout the paper, $\mathbb{R}^n_+$, $n\geq 1$, stands for the hyperoctant of $\mathbb{R}^n$ consisting of vectors $\textit{\textbf{s}} = (s_1,s_2,...,s_n)$  with the nonnegative coordinates $s_i\geq 0, \ i\in\overline{1,n}.$ $\mathbb{0}=(0,0,...,0)$ and $\mathbb{1} = (1,1,...,1)$ denote the origin and the unit vectors in $\mathbb{R}^n$ respectively. $||\cdot||$ is the Euclidean norm in $\mathbb{R}^n$.

In what follows, $S_n(u), \ u>0,$ is the centred $n$-dimensional hypersphere $\{\textit{\textbf{x}}\in\mathbb{R}^n: ||\textit{\textbf{x}}|| =u \}$, and $P_n[\mathbb{0},\textit{\textbf{t}}]$ is the hyperparallelepiped with the opposite vertices $\mathbb{0}$ and $\textit{\textbf{t}} = (t_1, t_2,..., t_n), \ t_i\in[0,1], \ i\in\overline{1,n}.$  $C$ with subindices represent generic finite positive constants, which are not necessarily same in each appearance.

In the following, we assume that all random variables and random fields are defined on the same probability space $\big\{\Omega, \mathcal{F}, P\big\}.$

\begin{assumption}
\label{assump_main}
Let $ \Lambda(\textit{\textbf{s}}),\ \textit{\textbf{s}} \in \mathbb{R}^n,$ be a measurable, homogeneous and isotropic, nonnegative random field  such that $P(\Lambda(\mathbb{0})>0)=1,$ $E\Lambda(\mathbb{0}) = 1,$ and $E\Lambda^2(\mathbb{0})<+\infty.$
\end{assumption}
\begin{remark}
The homogeneity and the isotropy are considered in the weak sense unless \mbox{otherwise stated}, i.e. the covariance function $E(\Lambda(\textit{\textbf{u}}_1)-1)(\Lambda(\textit{\textbf{u}}_2)-1), \ \textit{\textbf{u}}_1, \textit{\textbf{u}}_2 \in \mathbb{R}^n,$ is invariant with respect to the groups of motion and rotation transformations respectively.
\end{remark}
Let  $\Lambda^{(i)}(\cdot), \ i\in0,1,...,$ be an infinite collection of independent stochastic copies of $\Lambda(\cdot)$. Let $b > 1$ be a scaling parameter. It will be used to make homothetic transformations of~$\mathbb{R}^n.$ For $m\in\mathbb{N}$, the finite product $\Lambda_m(\cdot)$ of the random fields $\Lambda^{(i)}(\cdot)$ is defined by

\[\Lambda_m(\textit{\textbf{s}}) := \displaystyle\prod_{i=0}^{m-1} {\Lambda}^{(i)}(b^i\textit{\textbf{s}}), \]
and the cumulative random field $A_m(\cdot)$ is given by

\[A_m(\textit{\textbf{t}}) := \int_{P_n[\mathbb{0},\textit{\textbf{t}}]}\Lambda_m(\textit{\textbf{s}}) d\textit{\textbf{s}}, \ \textit{\textbf{t}} \in P_n[\mathbb{0},\mathbb{1}].\]

For a fixed $\textit{\textbf{t}}$, the sequence of random variables $\{A_m(\textit{\textbf{t}}), \ m\geq 1\}$ is a martingale with respect to the filtration $\mathcal{F}_m =\sigma\{\Lambda^{(0)}(\textit{\textbf{s}}),\Lambda^{(1)}(b\textit{\textbf{s}}),...,\Lambda^{(m-1)}(b^{m-1}\textit{\textbf{s}}),\textit{\textbf{s}}\in P_n[\mathbb{0},\textit{\textbf{t}}]\}, \ m\geq 1$. Indeed, for $m\geq j,$ by using Tonelli's theorem and the independence of $\Lambda^{(i)}(\cdot)$ for different~$i,$ one gets
\[E\left(A_m(\textit{\textbf{t}}) |  \mathcal{F}_j \right) =  E\left(\int_{P_n[\mathbb{0},\textit{\textbf{t}}]}\Lambda_m(\textit{\textbf{s}}) d\textit{\textbf{s}} \ \bigg| \  \mathcal{F}_j \right) = \int_{P_n[\mathbb{0},\textit{\textbf{t}}]} E\left(\displaystyle\prod_{i=0}^{m-1} \Lambda^{(i)}(b^i\textit{\textbf{s}}) \ \bigg| \ \mathcal{F}_j \right) d\textit{\textbf{s}} \]

\[ = \int_{P_n[\mathbb{0},\textit{\textbf{t}}]} \displaystyle\prod_{i=0}^{j-1} \Lambda^{(i)}(b^i\textit{\textbf{s}}) d\textit{\textbf{s}}  = A_j(\textit{\textbf{t}}).\]
The product $\Lambda_m(\cdot)$ can also be used to define the nonnegative random measures $\mu_m(\cdot)$ on Borel subsets $B\subseteq P_n[\mathbb{0} ,\mathbb{1}]$ as

\[ \mu_m(B) := \int_B\Lambda_m(\textit{\textbf{s}})d \textit{\textbf{s}}, \ m\in\mathbb{N}.\]

Let $\mu(\cdot)$ be a random measure defined on Borel subsets of $P_n[\mathbb{0},\mathbb{1}].$ The R{\'e}nyi function of the random measure $\mu(\cdot)$ is a deterministic function given by
\[
T(q) = \liminf\limits_{j\to\infty} \frac{\log_2E\sum_l\mu\left(  B_l^{(j)}  \right)^q}{\log_2\left|  B_0^{(j)}  \right|} = \liminf\limits_{j\to\infty}-\frac{\log_2E\sum_l\mu(B^{(j)}_l)^q}{nj}, \  q>0,
\] where  $\{B^{(j)}_l, \ l=0,1,...2^{nj}-1,$ $j = 1,2,...,\}$ denotes the mesh formed by the $j$-th level dyadic decomposition of $P_n[\mathbb{0},\mathbb{1}].$

\section{Limit theorems for multifractal products of random fields}
\label{sec2}
This section investigates the convergence of the random variables $A_m(\textit{\textbf{t}}), \ \textit{\textbf{t}}\in P_n[\mathbb{0},\mathbb{1}],$ in the spaces $L_q,$ when $m\to\infty.$

\begin{assumption}
\label{assump2}
Let $\textit{\textbf{p}} = (p_1,p_2,...,p_k),\ p_j\geq0, \ j=\overline{1,k}, \ k\geq2,$ and the function
\[\rho(\textit{\textbf{u}}_1, \textit{\textbf{u}}_2,...,\textit{\textbf{u}}_k,\textit{\textbf{p}} )  := E\bigg(\prod_{j=1}^k \Lambda^{p_j}(\textit{\textbf{u}}_j)\bigg),\] for all $\textit{\textbf{u}}_j \in \mathbb{R}^n,\ j = \overline{1,k}.$ Let also the function $\rho(\cdot)$ satisfy the condition
\[\rho(\textit{\textbf{u}}^{(1)}_1, \textit{\textbf{u}}^{(1)}_2,...,\textit{\textbf{u}}^{(1)}_k,\textit{\textbf{p}} ) \geq \rho(\textit{\textbf{u}}^{(2)}_1, \textit{\textbf{u}}^{(2)}_2,...,\textit{\textbf{u}}^{(2)}_k,\textit{\textbf{p}} )\] if $||\textit{\textbf{u}}^{(2)}_i - \textit{\textbf{u}}^{(2)}_j||\geq ||\textit{\textbf{u}}^{(1)}_i - \textit{\textbf{u}}^{(1)}_j||,$  $\textit{\textbf{u}}_1^{(l)},\textit{\textbf{u}}_2^{(l)},...,\textit{\textbf{u}}_k^{(l)}\in\mathbb{R}^n,\ l =1,2.$

\end{assumption}

\begin{remark}
\label{remark_assump} The function $\rho(\cdot)$ can be considered as a mixed moment, or a generalized $k$-point covariance function.
Assumption {\rm \ref{assump2}} is a mixing condition on the random field $\Lambda(\cdot)$. By this assumption, the $\rho$-dependence between the locations $\textit{\textbf{u}}^{(1)}_1, \textit{\textbf{u}}^{(1)}_2,...,\textit{\textbf{u}}^{(1)}_k$ is stronger than between the locations $\textit{\textbf{u}}^{(2)}_1, \textit{\textbf{u}}^{(2)}_2,...,\textit{\textbf{u}}^{(2)}_k$ if $\textit{\textbf{u}}^{(1)}_1, \textit{\textbf{u}}^{(1)}_2,...,\textit{\textbf{u}}^{(1)}_k$ are closer to each other than  $\textit{\textbf{u}}^{(2)}_1, \textit{\textbf{u}}^{(2)}_2,...,\textit{\textbf{u}}^{(2)}_k.$

\end{remark}

\begin{example}
\label{ex1}
 Let $\textit{\textbf{p}} = (p_1,p_2,...,p_k) $ such that $\sum_{j=1}^kp_j = p$, $X(\cdot)$ be a homogeneous and isotropic zero-mean Gaussian random field which covariance function $r_X(||\textit{\textbf{u}}||) = E(X(\mathbb{0})X(\textit{\textbf{u}})), \textit{\textbf{u}} \in\mathbb{R}^n,$ is nonincreasing in $||\textit{\textbf{u}}||.$ If $\Lambda(\textit{\textbf{s}}) = e^{X(\textit{\textbf{s}})}/Ee^{X(\mathbb{0})},$ $\textit{\textbf{s}}\in\mathbb{R}^n,\  n\geq 1,$  then Assumption {\rm\ref{assump2}} is satisfied.

It can be shown by the direct calculation of the expectation $E\bigg(\prod_{j=1}^k \Lambda^{p_j}(\textit{\textbf{u}}_j)\bigg).$
Indeed, as $X(\cdot)$ is a Gaussian field, the linear combination $\sum_{j=1}^k p_jX(\textit{\textbf{u}}_j)$ has a Gaussian distribution, and the random variable $e^{\sum_{j=1}^k p_jX(\textit{\textbf{u}}_j)}$ is log-normal. Thus,

\[E\bigg(\prod_{j=1}^k \Lambda^{p_j}(\textit{\textbf{u}}_j)\bigg) = \frac{1}{\left(Ee^{X(0)}\right)^p} Ee^{\sum_{j=1}^k p_jX(\textit{\textbf{u}}_j)}=\frac{1}{e^{\frac{p}{2}EX^2(\mathbb{0})}} e^{{\frac{1}{2}E(\sum_{j=1}^k p_jX(\textit{\textbf{u}}_j))^2}}\]

\begin{equation}
\label{moment}
=\frac{1}{e^{\frac{p}{2}EX^2(\mathbb{0})}} \exp\bigg(\frac{1}{2}\bigg( \sum_{j=1}^kp^2_jEX^2(\mathbb{0}) + \sum\limits_{\substack{i,j=1 \\ i \neq j}}^kp_ip_jr_X(||\textit{\textbf{u}}_i-\textit{\textbf{u}}_j||)\bigg) \bigg).
\end{equation}

If $\{ \textit{\textbf{u}}^{(1)}_1, \textit{\textbf{u}}^{(1)}_2,...,\textit{\textbf{u}}^{(1)}_k \}$ and $\{\textit{\textbf{u}}^{(2)}_1, \textit{\textbf{u}}^{(2)}_2,...,\textit{\textbf{u}}^{(2)}_k\}$  are two sets of points from $\mathbb{R}^n$ satisfying the inequalities $||\textit{\textbf{u}}^{(2)}_i - \textit{\textbf{u}}^{(2)}_j||\geq ||\textit{\textbf{u}}^{(1)}_i - \textit{\textbf{u}}^{(1)}_j||,\ i,j = \overline{1,k},$ then, $r_X(||\textit{\textbf{u}}^{(2)}_i-\textit{\textbf{u}}^{(2)}_j||) \leq r_X(||\textit{\textbf{u}}^{(1)}_{i}-\textit{\textbf{u}}^{(1)}_{j}||)$ because the function $r_X(\cdot)$ is nonincreasing. Thus, Assumption {\rm \ref{assump2}} follows from \eqref{moment}.

\end{example}

For the detailed discussion of the log-normal and other scenarios see Section \ref{sec4}.

If $\textit{\textbf{p}}=(1,1,...,1)$ is a vector with all components equal 1, then for the simplisity of the exposition, we will use the notation
\[\rho(\textit{\textbf{u}}_1, \textit{\textbf{u}}_2,...,\textit{\textbf{u}}_k):=E\bigg(\prod_{j=1}^k\Lambda(\textit{\textbf{u}}_j)\bigg). \]Note, that due to the homogeneity and the isotropy of $\Lambda(\cdot),$ $\rho(\textit{\textbf{u}}_1, \textit{\textbf{u}}_2) = \rho(||\textit{\textbf{u}}_1- \textit{\textbf{u}}_2||).$

\begin{definition}
A random field $X(\textit{\textbf{u}}), \ \textit{\textbf{u}}\in\mathbb{R}^n,$ is $k$-weakly associated, if for any integers $l$ and $l^{'},\ 1\leq l^{'}<l\leq k,$ and $\textit{\textbf{u}}_1, \textit{\textbf{u}}_2,...,\textit{\textbf{u}}_l\in\mathbb{R}^n,$ it holds
\[ cov\left( X(\textit{\textbf{u}}_1)\cdot...\cdot X(\textit{\textbf{u}}_{l^{'}}), X(\textit{\textbf{u}}_{l^{'}+1})\cdot...\cdot X(\textit{\textbf{u}}_{l}) \right)\geq 0. \]
\end{definition}

\begin{remark}
The class of $k$-weakly associated processes includes the class of associated ones (see the corresponding definitions and results on associated random variables in {\rm \cite{rao2012associated}}).
\end{remark}

\begin{remark}
\label{remark_assoc_main}
If $X(\textit{\textbf{u}}), \ \textit{\textbf{u}}\in\mathbb{R}^n,$ is a $k$-weakly associated random field satisfying Assumption {\rm \ref{assump_main}}, then for any integer $l, \ 1\leq l \leq k,$ and any $\textit{\textbf{u}}_1, \textit{\textbf{u}}_2,...,\textit{\textbf{u}}_l\in\mathbb{R}^n,$

\begin{equation}
\label{assoc_ineq}
\rho(\textit{\textbf{u}}_1, \textit{\textbf{u}}_2,...,\textit{\textbf{u}}_l)\geq 1.
\end{equation} Indeed, for $l=2$ and $l'=1$ it follows from $k$-weakly association and Assumption  {\rm \ref{assump_main}}, that
\[ cov(X(\textit{\textbf{u}}_1),X(\textit{\textbf{u}}_2)) = EX(\textit{\textbf{u}}_1)X(\textit{\textbf{u}}_2) - EX(\textit{\textbf{u}}_1)EX(\textit{\textbf{u}}_2)=\rho(\textit{\textbf{u}}_1,\textit{\textbf{u}}_2)-1\geq 0.\] By increasing $l$ by $1$ the statement \eqref{assoc_ineq} follows recursively.

Moreover, for any $l, \ 2\leq l \leq k,$
\[ \rho(\textit{\textbf{u}}_1, \textit{\textbf{u}}_2,...,\textit{\textbf{u}}_l) \geq \rho(\textit{\textbf{u}}_1, \textit{\textbf{u}}_2,...,\textit{\textbf{u}}_{l-1}),\] as
\[ cov(X(\textit{\textbf{u}}_1)X(\textit{\textbf{u}}_{2})\cdot...\cdot X(\textit{\textbf{u}}_{l-1}),X(\textit{\textbf{u}}_l)) = \rho(\textit{\textbf{u}}_1, \textit{\textbf{u}}_2,...,\textit{\textbf{u}}_l) - \rho(\textit{\textbf{u}}_1, \textit{\textbf{u}}_2,...,\textit{\textbf{u}}_{l-1}) \geq 0. \]
\end{remark}

For the simplicity of the exposition, in what follows $\left(P_n[\mathbb{0},\textit{\textbf{t}}]\right)^k, \ k\in\mathbb{N},$ stands for the Cartesian product $P_n[\mathbb{0},\textit{\textbf{t}}]\times P_n[\mathbb{0},\textit{\textbf{t}}]\times...\times P_n[\mathbb{0},\textit{\textbf{t}}]$ of $k$ sets $P_n[\mathbb{0},\textit{\textbf{t}}].$

\begin{theorem}
\label{th2}
Let Assumption {\rm \ref{assump_main}} be satisfied, Assumption {\rm \ref{assump2}} holds true for the vector $\textit{\textbf{p}} = (p_1,p_2,...,p_k),$ $p_j\geq1, \ j=\overline{1,k},$ such that $\sum_{j=1}^kp_j = p \geq2,$
\begin{equation}
\label{cond2.1}
b>(E\Lambda^p(\mathbb{0}))^{\frac{1}{n}}
\end{equation} and
\begin{equation}
\label{cond2.2}
\sum_{i=0}^\infty\ln\left(\rho(\mathbb{0}, b^i\mathbb{1}, 2b^i\mathbb{1},...,(k-1)b^i\mathbb{1}, \textit{\textbf{p}})\right) < \infty.
\end{equation}  Then, for every $\textit{\textbf{t}} \in P_n[\mathbb{0},\mathbb{1}]$ and for all  $q\in[0,p],$ the random variables $A_m(\textit{\textbf{t}})$  converge to some random variables $A(\textit{\textbf{t}})$ in the spaces $L_q,$ as $m \to \infty$.

\end{theorem}

\begin{proof}

Being a martingale, the sequence $\{A_m(\textit{\textbf{t}}), m \geq 1\}$ converges almost surely and in the space $L_p,$ as $m \to \infty,$ if $sup_{m}EA_{m}^p(\textit{\textbf{t}})<\infty,$ see {\rm\cite[Proposition {\rm IV-2-7}]{Neveu}}. Thus, it is enough to show that for all $m\geq 1$ the expectations $EA_{m}^p(\textit{\textbf{t}})$ are bounded by the same constant.

If $x_1,x_2\in\mathbb{R}^+,\ x_1<x_2$, then $||x_2\textit{\textbf{s}}_i-x_2\textit{\textbf{s}}_j||\geq||x_1\textit{\textbf{s}}_i -x_1\textit{\textbf{s}}_j||, \ i,j=\overline{1,k},$ for all $\textit{\textbf{s}}_1,\textit{\textbf{s}}_2,...,\textit{\textbf{s}}_k\in\mathbb{R}^n.$ Thus, by Assumption \ref{assump2} it follows that for all fixed points $\textit{\textbf{s}}_1,\textit{\textbf{s}}_2,...,\textit{\textbf{s}}_k \in \mathbb{R}^n$ the function $\rho(x\textit{\textbf{s}}_1,x\textit{\textbf{s}}_2,...,x\textit{\textbf{s}}_k, \textit{\textbf{p}})$ is nonincreasing in $x\in\mathbb{R}^+$.

Now let us consider a moment of order $p$ of $A_m(\textit{\textbf{t}})$
\[E(A_m(\textit{\textbf{t}}))^p =  E \prod_{j=1}^k\left(\int\limits_{P_n[\mathbb{0},\textit{\textbf{t}}]}\Lambda_m(\textit{\textbf{s}}_j)d\textit{\textbf{s}}_j\right)^{p_j}.\] As $p_i\geq1, \ i=\overline{1,k},$ one can apply H\"{o}lder's inequality and obtain the upper bound
\[E(A_m(\textit{\textbf{t}}))^p \leq C \cdot E \left(\prod_{j=1}^k\int\limits_{P_n[\mathbb{0},\textit{\textbf{t}}]}\Lambda_m^{p_j}(\textit{\textbf{s}}_j)d\textit{\textbf{s}}_j\right) = C \cdot E\left( \int\limits_{\big(P_n[\mathbb{0},\textit{\textbf{t}}]\big)^k}  \prod_{j=1}^k \big(\Lambda_m(\textit{\textbf{s}}_j)\big)^{p_j} \prod_{j=1}^kd\textit{\textbf{s}}_j \right)\]
\[= C\cdot E \left( \int\limits_{\big(P_n[\mathbb{0},\textit{\textbf{t}}]\big)^k}\displaystyle\prod_{i=0}^{m-1}\prod_{j=1}^k \bigg( \Lambda^{(i)}(b^i\textit{\textbf{s}}_j) \bigg)^{p_j} \prod_{j=1}^k d\textit{\textbf{s}}_j \right) \]
\begin{equation}\label{rho} =  C \int\limits_{\left(P_n[\mathbb{0},\textit{\textbf{t}}]\right)^k} \prod_{i= 0}^{m-1} \rho(b^i\textit{\textbf{s}}_1,b^i\textit{\textbf{s}}_2,...,b^i\textit{\textbf{s}}_k, \textit{\textbf{p}} ) \prod_{j=1}^k d\textit{\textbf{s}}_j.
\end{equation}

Let us majorize the terms in the above integral. Consider $k$ equidistant points $ \textit{\textbf{u}}_j^{(1)}, \ j=\overline{1,k},$ on the vector  from $\mathbb{0}$ to $\frac{b^i}{\sqrt{n}}\min\limits_{\substack{l,h: l\neq h}}||\textit{\textbf{s}}_l-\textit{\textbf{s}}_h|| \mathbb{1}$ that are
\[ \textit{\textbf{u}}_j^{(1)} = \frac{(j-1)b^i\min\limits_{\substack{l,h: l\neq h}}||\textit{\textbf{s}}_l-\textit{\textbf{s}}_h|| }{(k-1)\sqrt{n}} \mathbb{1}, \ \ j=\overline{1,k},\] and the set of points $\textit{\textbf{u}}_j^{(2)} = b^i\textit{\textbf{s}}_j,\ j=\overline{1,k}.$

As
\[b^i \min\limits_{\substack{l,h: l\neq h}}||\textit{\textbf{s}}_l-\textit{\textbf{s}}_h||  = \frac{b^i\min\limits_{\substack{l,h: l\neq h}}||\textit{\textbf{s}}_l-\textit{\textbf{s}}_h|| }{\sqrt{n}} \left|\left|\mathbb{1}\right|\right|\]
it is easy to see that for all $i,j = \overline{1,k},$ it holds
$||\textit{\textbf{u}}^{(2)}_i - \textit{\textbf{u}}^{(2)}_j||\geq ||\textit{\textbf{u}}^{(1)}_i - \textit{\textbf{u}}^{(1)}_j||.$
By this inequality and Assumption \ref{assump2}
\[\rho(b^i\textit{\textbf{s}}_1,b^i\textit{\textbf{s}}_2,...,b^i\textit{\textbf{s}}_k, \textit{\textbf{p}} ) \] \[\leq \rho\bigg(\mathbb{0}, \frac{\min\limits_{\substack{l,h: l\neq h}}||\textit{\textbf{s}}_l-\textit{\textbf{s}}_h|| b^i\mathbb{1}}{(k-1)\sqrt{n}},\frac{2\min\limits_{\substack{l,h: l\neq h}}||\textit{\textbf{s}}_l-\textit{\textbf{s}}_h|| b^i\mathbb{1}}{(k-1)\sqrt{n}},...,\frac{(k-1)\min\limits_{\substack{l,h: l\neq h}}||\textit{\textbf{s}}_l-\textit{\textbf{s}}_h|| b^i\mathbb{1}}{(k-1)\sqrt{n}}, \textit{\textbf{p}} \bigg).\]
Thus, $E(A_m(\textit{\textbf{t}}))^q$ can be estimated from above by
\[C \int\limits_{\left(P_n[\mathbb{0},\textit{\textbf{t}}]\right)^k} \prod_{i=0}^{m-1}\rho\bigg(\mathbb{0}, \frac{1\cdot b^i\min\limits_{\substack{l,h: l\neq h}}||\textit{\textbf{s}}_l-\textit{\textbf{s}}_h||\mathbb{1}}{(k-1)\sqrt{n}},...,\frac{(k-1)\cdot b^i\min\limits_{\substack{l,h: l\neq h}}||\textit{\textbf{s}}_l-\textit{\textbf{s}}_h||\mathbb{1}}{(k-1)\sqrt{n}}, \textit{\textbf{p}}\bigg)  \prod_{j=1}^k d\textit{\textbf{s}}_j \]
\[ \leq C \sum\limits_{\substack{l,h: l\neq h}}  \int\limits_{\left(P_n[\mathbb{0},\textit{\textbf{t}}]\right)^k} \prod_{i=0}^{m-1}\rho\bigg(\mathbb{0}, \frac{1\cdot b^i||\textit{\textbf{s}}_l-\textit{\textbf{s}}_h||\mathbb{1}}{(k-1)\sqrt{n}},...,\frac{(k-1)\cdot  b^i||\textit{\textbf{s}}_l-\textit{\textbf{s}}_h||\mathbb{1}}{(k-1)\sqrt{n}}, \textit{\textbf{p}}\bigg)  \prod_{j=1}^k d\textit{\textbf{s}}_j \]
\[ \leq C \bigg(\prod\limits_{i=1}^n t_i \bigg)^{k-2} \sum\limits_{\substack{l,h: l\neq h}}  \int\limits_{\left(P_n[\mathbb{0},\textit{\textbf{t}}]\right)^2} \prod_{i=0}^{m-1}\rho\bigg(\mathbb{0}, \frac{1\cdot b^i||\textit{\textbf{s}}_l-\textit{\textbf{s}}_h||\mathbb{1}}{(k-1)\sqrt{n}},...,\frac{(k-1)\cdot  b^i||\textit{\textbf{s}}_l-\textit{\textbf{s}}_h||\mathbb{1}}{(k-1)\sqrt{n}}, \textit{\textbf{p}}\bigg) d\textit{\textbf{s}}_l d\textit{\textbf{s}}_h\]
\[ = C \bigg(\prod\limits_{i=1}^nt_i\bigg)^{k-2} \int\limits_{\left(P_n[\mathbb{0},\textit{\textbf{t}}]\right)^2} \prod_{i=0}^{m-1}\rho\bigg(\mathbb{0}, \frac{1\cdot b^i||\textit{\textbf{s}}_1-\textit{\textbf{s}}_2||\mathbb{1}}{(k-1)\sqrt{n}},...,\frac{(k-1)\cdot b^i ||\textit{\textbf{s}}_1-\textit{\textbf{s}}_2||\mathbb{1}}{(k-1)\sqrt{n}}, \textit{\textbf{p}}\bigg)  d\textit{\textbf{s}}_1d\textit{\textbf{s}}_2.\]

In what follows $A \circleddash B$ denotes the Minkowski difference of the sets $A$ and $B.$ By the change of variables $\textit{\textbf{s}} = \textit{\textbf{s}}_1 - \textit{\textbf{s}}_2$  in the above integral, one obtains

\[  C \bigg(\prod\limits_{i=1}^nt_i\bigg)^{k-1} \int\limits_{P_n[\mathbb{0},\textit{\textbf{t}}]\circleddash P_n[\mathbb{0},\textit{\textbf{t}}]} \prod_{i=0}^{m-1}\rho\bigg(\mathbb{0}, \frac{b^i||\textit{\textbf{s}}||\mathbb{1}}{(k-1)\sqrt{n}},\frac{2 b^i ||\textit{\textbf{s}}||\mathbb{1}}{(k-1)\sqrt{n}},...,\frac{(k-1) b^i ||\textit{\textbf{s}}||\mathbb{1}}{(k-1)\sqrt{n}}, \textit{\textbf{p}}\bigg)  d\textit{\textbf{s}}.\] Using the hyperspherical coordinates, one gets that the last expression is bounded from \mbox{above by}

\[ C \int\limits_{0}^{\sqrt{n}} u^{n-1} \prod_{i=0}^{m-1}\rho\bigg(\mathbb{0}, \frac{b^iu\mathbb{1}}{(k-1)\sqrt{n}},\frac{2 b^i u \mathbb{1}}{(k-1)\sqrt{n}},...,\frac{(k-1) b^i u \mathbb{1}}{(k-1)\sqrt{n}}, \textit{\textbf{p}}\bigg)  du.\] By the change of variables $u = (k-1)\sqrt{n}x,$

\begin{equation}\label{int_prod} E(A_m(\textit{\textbf{t}}))^p \leq C \int\limits_{0}^{1/(k-1)}x^{n-1}  \prod_{i=1}^{m-1}\rho\big(\mathbb{0}, b^i x \mathbb{1},2b^i x \mathbb{1},...,(k-1)b^i x \mathbb{1}, \textit{\textbf{p}}\big)  dx.
\end{equation}

Since $b>1$ and $\rho\big(\mathbb{0}, b^i x \mathbb{1},2b^i x \mathbb{1},...,(k-1)b^i x \mathbb{1}, \textit{\textbf{p}}\big)$ is nonincreasing in $i$, the product in \eqref{int_prod} can be estimated as follows
\[\prod_{i=1}^{m-1}\rho\big(\mathbb{0}, b^i x \mathbb{1},2b^i x \mathbb{1},...,(k-1)b^i x \mathbb{1}, \textit{\textbf{p}}\big) \]
\[ = \exp\bigg( \sum_{i=1}^{m-1}\ln \big(\rho\big(\mathbb{0}, b^i x \mathbb{1},2b^i x \mathbb{1},...,(k-1)b^i x \mathbb{1}, \textit{\textbf{p}}\big)\big)\bigg) \]
\[\leq \exp\bigg(\int_{0}^{m-1} \ln \big(\rho\big(\mathbb{0}, b^y x \mathbb{1},2b^y x \mathbb{1},...,(k-1)b^y x \mathbb{1}, \textit{\textbf{p}}\big)\big) dy\bigg)\]
\[ =  \exp\bigg(\int_{\log_bx}^{m-1+\log_bx}\ln \big(\rho\big(\mathbb{0}, b^u \mathbb{1},2b^u \mathbb{1},...,(k-1)b^u  \mathbb{1}, \textit{\textbf{p}}\big)\big) du\bigg),\] where the last equality is obtained by the change of variables $y = u-\log_bx.$

Thus, \eqref{int_prod} is bounded from above by
\begin{equation}\label{int_int}  C \int\limits_{0}^{1/(k-1)}x^{n-1} \exp\bigg(\int_{\log_bx}^{m-1+\log_bx}\ln \big(\rho\big(\mathbb{0}, b^u \mathbb{1},2b^u   \mathbb{1},...,(k-1)b^u  \mathbb{1}, \textit{\textbf{p}}\big)\big) du\bigg)dx.
\end{equation}
Let us consider two separate cases. First, assume that there exists a finite $x_0$ such that
$\ln \big(\rho\big(\mathbb{0}, b^{x_0} \mathbb{1},2b^{x_0} \mathbb{1},...,(k-1)b^{x_0}  \mathbb{1}, \textit{\textbf{p}}\big)\big)=0.$ Note that in that case $\ln \big(\rho\big(\mathbb{0}, b^{x } \mathbb{1},2b^{x } \mathbb{1},...,(k-1)b^{x }  \mathbb{1}, \textit{\textbf{p}}\big)\big)\leq0$ for all $x\geq x_0$ as $\rho(\cdot)$ is nonincreasing. Then, in the neighbourhood of $0$ the internal integral in \eqref{int_int} can be \mbox{estimated by}

\[\int_{\log_bx}^{x_0}\ln \big(\rho\big(\mathbb{0}, b^u \mathbb{1},2b^u   \mathbb{1},...,(k-1)b^u  \mathbb{1}, \textit{\textbf{p}}\big)\big)du \leq \ln\left(E\Lambda^p(\mathbb{0})\right)\left(x_0-\log_bx\right)\] and the integrand in \eqref{int_int} is bounded \mbox{from above by}
\[x^{n-1}(E\Lambda^p(\mathbb{0}))^{x_0-\log_b  x}=Cx^{n-1}(E\Lambda^p(\mathbb{0}))^{-\log_bx}.\] By applying the identity $a^{\log_bc}=c^{\log_ba},$ one gets $C x^{n-1-\log_bE\Lambda^p(\mathbb{0})},$ which means that \eqref{int_int} is finite if  $b>(E\Lambda^p(\mathbb{0}))^{\frac{1}{n}},$ which provides the condition \eqref{cond2.1} of the theorem.

Lets consider the case when $\ln \big(\rho\big(\mathbb{0}, b^u \mathbb{1},2b^u x \mathbb{1},...,(k-1)b^u  \mathbb{1}, \textit{\textbf{p}}\big)\big)>0$ for all $u\in\mathbb{R}.$ Then, \eqref{int_int} allows the following estimation from above
\[C \int\limits_{0}^{1/(k-1)}x^{n-1} \exp\bigg(\int_{\log_bx}^{\infty}\ln \big(\rho\big(\mathbb{0}, b^u \mathbb{1},2b^u\mathbb{1},...,(k-1)b^u  \mathbb{1}, \textit{\textbf{p}}\big)\big) du\bigg)dx.\]
 As $x\in[0,1/(k-1)],$ then $\log_bx<0,$ and the above integral can be rewritten as
\[ C \int\limits_{0}^{1/(k-1)}x^{n-1} \exp\bigg(\int_{\log_bx}^{0}\ln \big(\rho\big(\mathbb{0}, b^u \mathbb{1},2b^u\mathbb{1},...,(k-1)b^u  \mathbb{1}, \textit{\textbf{p}}\big)\big) du\bigg)\]
\begin{equation}
\label{eq2}
\times\exp\bigg(\int_{0}^{\infty}\ln \big(\rho\big(\mathbb{0}, b^u \mathbb{1},2b^u\mathbb{1},...,(k-1)b^u  \mathbb{1}, \textit{\textbf{p}}\big)\big) du\bigg)dx.
\end{equation}
The function $\rho(\mathbb{0}, b^u\mathbb{1}, 2b^u\mathbb{1},...,(k-1)b^u\mathbb{1}, \textit{\textbf{p}}))$ is decreasing in $u.$ Hence, the following estimate holds true
\[ \int_{0}^\infty \ln (\rho(\mathbb{0}, b^u\mathbb{1}, 2b^u\mathbb{1},...,(k-1)b^u\mathbb{1}, \textit{\textbf{p}}))du \leq  \sum_{i = 0}^\infty  \ln (\rho(\mathbb{0}, b^i\mathbb{1}, 2b^i\mathbb{1},...,(k-1)b^i\mathbb{1}, \textit{\textbf{p}})).\] The last series is finite due to condition \eqref{cond2.2}.
Therefore, the second exponent in \eqref{eq2} is finite.

Thus, $E(A_m(\textit{\textbf{t}}))^p$ is bounded from above by the integral
\[ C \int\limits_{0}^{1/(k-1)}x^{n-1} \exp\bigg(\int_{\log_bx}^{0}\ln \big(\rho\big(\mathbb{0}, b^u \mathbb{1},2b^u  \mathbb{1},...,(k-1)b^u  \mathbb{1}, \textit{\textbf{p}}\big)\big)du\bigg)dx .\]
By the same reasons as in the first case this integral is bounded.  Thus, $A_m(\textit{\textbf{t}})$ converges in the space $L_p,$ as $m\to\infty.$

The convergence in the spaces $L_q, \ q\in[0,p]$, follows by Jensen's inequality, which finishes the proof. \end{proof}

\begin{remark}\label{limits}
The conditions of Theorem {\rm\ref{th2}} also guarantee almost sure convergences of random variables $A_{m}(\textit{\textbf{t}}),\ \textit{\textbf{t}}\in P_n[\mathbb{0}, \mathbb{1}],$ as $m\to\infty$. Moreover, almost sure limits  and limits in the spaces $L_q,\ q\in[0,p],$  coincide, see {\rm\cite[Proposition {\rm IV-2-7}]{Neveu}}.
\end{remark}

\begin{remark}
If there exists $x_0$ such that $\ln(\rho(\mathbb{0},b^{x_0},2b^{x_0}\mathbb{1},...,(k-1)b^{x_0}\mathbb{1},\textit{\textbf{p}}))=0,$ then condition \eqref{cond2.2} is satisfied. Indeed,  the fulfilment of \eqref{cond2.2} follows from the inequality  $\ln(\rho(\mathbb{0},b^{x},2b^{x}\mathbb{1},...,(k-1)b^{x}\mathbb{1},\textit{\textbf{p}}))\leq0, \ x\geq x_0,$ which holds as  $\rho(\cdot)$ is a nonincreasing function.
\end{remark}

\begin{corollary}
\label{cor_lq_unit}
Let the assumptions of Theorem {\rm \ref{th2}} hold true for the vector $\textit{\textbf{p}} = (1,1,...,1)$ that has $ {p}$ components equal to $1$. Then, for every $\textit{\textbf{t}} \in P_n[\mathbb{0},\mathbb{1}]$ and for all $q\in[0,p],$ the random variables $A_m(\textit{\textbf{t}})$  converge to some random variable $A(\textit{\textbf{t}})$ in the spaces $L_q,$ as $m \to \infty$.

\end{corollary}

A partial case is the result about $L_2$ convergence:

\begin{corollary}
\label{th1}
 Let Assumption {\rm \ref{assump_main}}  be satisfied,  $\Lambda(\cdot)$ be a random field such that the function $\rho(||\textit{\textbf{u}}||) = E\Lambda(\mathbb{0})\Lambda(\textit{\textbf{u}}), \ \textit{\textbf{u}} \in \mathbb{R}^n,$ is nonincreasing in $||\textit{\textbf{u}}||$,
\[
b > \left(E\Lambda^2(\mathbb{0})\right)^{1/n}
\] and
\[
\label{cond2}
\sum_{i = 0}^{\infty} \ln \left(\rho\left(\sqrt{n}b^{i}\right) \right) < \infty.
\]
 Then, for every $\textit{\textbf{t}} \in P_n[\mathbb{0},\mathbb{1}]$ and for all $q\in[0,2],$ the random variables $A_m(\textit{\textbf{t}})$  converge to some random variable $A(\textit{\textbf{t}})$ in the spaces $L_q,$ as $m \to \infty$.
\end{corollary}

Theorem \ref{th2} also holds under more general conditions that will be used in the following sections.

\begin{assumption}
\label{assump2'}
Let $\textit{\textbf{p}} = (p_1,p_2,...,p_k),\ p_j\geq1, \ j=\overline{1,k},\ k\geq2,$ such that $\sum_{j=1}^kp_j = p \geq2,$  and there exist a function \[\widetilde{\rho}(\textit{\textbf{u}}_1, \textit{\textbf{u}}_2,...,\textit{\textbf{u}}_k,\textit{\textbf{p}}) \geq E\bigg(\prod_{j=1}^k \Lambda^{p_j}(\textit{\textbf{u}}_j)\bigg)\] that satisfies the inequality
\[\widetilde{\rho}(\textit{\textbf{u}}^{(1)}_1, \textit{\textbf{u}}^{(1)}_2,...,\textit{\textbf{u}}^{(1)}_k,\textit{\textbf{p}}) \geq \widetilde{\rho}(\textit{\textbf{u}}^{(2)}_1, \textit{\textbf{u}}^{(2)}_2,...,\textit{\textbf{u}}^{(2)}_k,\textit{\textbf{p}})\]
for all $\textit{\textbf{u}}_1^{(l)},\textit{\textbf{u}}_2^{(l)},...,\textit{\textbf{u}}_k^{(l)}\in\mathbb{R}^n,\ l =1,2, $ such that $||\textit{\textbf{u}}^{(2)}_i - \textit{\textbf{u}}^{(2)}_j||\geq ||\textit{\textbf{u}}^{(1)}_i - \textit{\textbf{u}}^{(1)}_j||,\ i,j = \overline{1,k}.$

\end{assumption}

For simplicity, in the case of $\textit{\textbf{p}}=(1,1,...,1)$ we will use the notation $\widetilde{\rho}(\textit{\textbf{u}}_1, \textit{\textbf{u}}_2,...,\textit{\textbf{u}}_k).$

\begin{corollary} \label{cor_for_sub}
Let Assumptions {\rm \ref{assump_main}} and {\rm \ref{assump2'}} hold true,
\begin{equation} \label{cond2.1'}
b>\widetilde{\rho}^{\frac{1}{n}}(\mathbb{0},\mathbb{0},...,\mathbb{0},\textit{\textbf{p}}),
\end{equation} and
\begin{equation}\label{cond2.2'}
\sum_{i=0}^\infty\ln\left(\widetilde{\rho}(\mathbb{0}, b^i\mathbb{1}, 2b^i\mathbb{1},...,(k-1)b^i\mathbb{1},\textit{\textbf{p}}) \right) < \infty.
\end{equation} Then, for every $\textit{\textbf{t}} \in P_n[\mathbb{0},\mathbb{1}]$ and for all $q\in[0,p],$ the random variables $A_m(\textit{\textbf{t}})$  converge to some random variable $A(\textit{\textbf{t}})$ in the spaces $L_q,$ as $m \to \infty$.
\end{corollary}

\begin{proof}
The proof is analogous to the proof of Theorem \ref{th2}. The key modification is in~\eqref{rho} where an additional  step is required. Namely, the integral in $\eqref{rho}$ can be bounded by
\[ \int\limits_{\left(P_n[\mathbb{0},\textit{\textbf{t}}]\right)^k} \prod_{i= 0}^{m-1} \rho(b^i\textit{\textbf{s}}_1,b^i\textit{\textbf{s}}_2,...,b^i\textit{\textbf{s}}_k, \textit{\textbf{p}} ) \prod_{j=1}^k d\textit{\textbf{s}}_j  \leq  \int\limits_{\left(P_n[\mathbb{0},\textit{\textbf{t}}]\right)^k} \prod_{i= 0}^{m-1} \widetilde{\rho}(b^i\textit{\textbf{s}}_1,b^i\textit{\textbf{s}}_2,...,b^i\textit{\textbf{s}}_k, \textit{\textbf{p}} ) \prod_{j=1}^k d\textit{\textbf{s}}_j.\]
As for all fixed points $\textit{\textbf{s}}_1,\textit{\textbf{s}}_2,...,\textit{\textbf{s}}_k \in \mathbb{R}^n$ the function $\widetilde{\rho}(x\textit{\textbf{s}}_1,x\textit{\textbf{s}}_2,...,x\textit{\textbf{s}}_k, \textit{\textbf{p}} )$ is decreasing in $x\in\mathbb{R}^+,$  the rest of the proof is the same.   \end{proof}

\begin{corollary}\label{measure} If the conditions of Theorem {\rm \ref{th2}} are satisfied, then, for a given finite or countable family of Borel sets $\mathfrak{B} = \{B_j: \ B_j\subseteq P_n[\mathbb{0} ,\mathbb{1}]\}$, it holds
\[ \lim\limits_{m\to\infty} \mu_m(B_j) = \mu(B_j) \ \ \ \ a.s.\]
\end{corollary}

\begin{proof}
Let $B_j\in \mathfrak{B}$. Then $\{\mu_m(B_j), \ m\geq 1\}$ is a martingale with respect to the filtration $\mathcal{F}_m = \sigma\{\Lambda^{(0)}(\textit{\textbf{s}}),\Lambda^{(1)}(b\textit{\textbf{s}}),...,\Lambda^{(m-1)}(b^{m-1}\textit{\textbf{s}}),\textit{\textbf{s}}\in B_j\}, \ m\geq 1$. Repeating the same steps as in the proof of Theorem \ref{th2}, one can show that $\lim\limits_{m\to\infty} \mu_m(B_j) = \mu(B_j)$ a.s. As the family $\mathfrak{B}$ is  finite or countable,  the almost sure convergence holds true for the whole \mbox{family $\mathfrak{B}.$}
\end{proof}
\begin{remark}
Under the assumptions of Corollary {\rm \ref{measure}} the convergence  $\lim\limits_{m\to\infty} \mu_m(B_j) = \mu(B_j)$ also holds true in the spaces $L_q, \ q\in[0,p],$ for each $\ B_j\in \mathfrak{B}.$
\end{remark}

\section{Rates of convergence in L\small{q}}
\label{rate_lq_sec}
This section provides the rates of convergence for the random variables $A_m(\cdot)$ and the random measures $\mu_m(\cdot)$ in the limit theorems from Section \ref{sec2}.

\begin{theorem}
\label{rate_theorem_lq_fast}
Let Assumptions {\rm \ref{assump_main}} and {\rm \ref{assump2}} hold true and $\rho(\cdot)\geq 1 $ for all vectors $\textit{\textbf{q}}' = (1,1,...,1)$ with $ q', \ 2 \leq q'\leq {p},$ components, where $p $ is an even integer.  Also, let there exist  such $x_0$ and $C_1>0$ that for all $q'$ and $x\geq x_0$
\begin{equation*}
\ln\rho(\mathbb{0},x\mathbb{1},2x\mathbb{1},...,(q'-1)x\mathbb{1}) \leq C_1 x^{-\alpha}, \ \alpha> \frac{\ln(E\Lambda^p(\mathbb{0}))}{\ln b}.\end{equation*} If
\begin{equation}
\label{rate_cond2_lq_fast}
b >  \left(E\Lambda^p(\mathbb{0})\right)^{\frac{\gamma}{(\gamma-1)n}}
\end{equation} for some $\gamma \in (\max\left(1, n/\alpha\right), n\ln(b)/\ln(E\Lambda^p(\mathbb{0}))$,
then for all $\textit{\textbf{t}}\in P_n[\mathbb{0},\mathbb{1}]$ and $q \in [2,p]$
\begin{equation*}
E|A(\textit{\textbf{t}})-A_m(\textit{\textbf{t}})|^q\leq C  \bigg(\frac{ \prod_{i=1}^nt_i^{q-1}}{b^{nm}}\bigg)^{1/\gamma}.
\end{equation*}
\end{theorem}

\begin{remark} \label{remark_assoc}
It follows from Remark {\rm \ref{remark_assoc_main}} that $1\leq \rho(\mathbb{0},x\mathbb{1},2x\mathbb{1},...,(q'-1)x\mathbb{1})\leq \rho(\mathbb{0},x\mathbb{1},2x\mathbb{1},\\...,(p-1)x\mathbb{1})$ for all $q', \ 2\leq q'\leq p,$ if the random field $\Lambda(\textit{\textbf{u}}), \ \textit{\textbf{u}}\in\mathbb{R}^n,$ is $p$-weakly associated.
\end{remark}

\begin{proof}
When $l\to\infty$, the random variables $A_l(\textbf{\textit{t}})$ converge to the random variable $A(\textbf{\textit{t}})$, and
\begin{equation}
\label{rate_lq_eq1_proof}
\lim\limits_{l\to\infty} E|A_l(\textit{\textbf{t}})-A_m(\textit{\textbf{t}})|^{p} = E|A(\textit{\textbf{t}})-A_m(\textit{\textbf{t}})|^{p}.
\end{equation}
For any $l>m,$ it holds
\[ E|A_l(\textit{\textbf{t}})-A_m(\textit{\textbf{t}})|^{p}\]
\[ = E \bigg(\int\limits_{P_n[\mathbb{0},\textit{\textbf{t}}]} \bigg( \displaystyle\prod_{i=0}^{l} \Lambda^{(i)}(b^i\textit{\textbf{s}}) - \displaystyle\prod_{i=0}^{m} \Lambda^{(i)}(b^i\textit{\textbf{s}}) \bigg) d\textit{\textbf{s}} \bigg)^{p} \] \[= E \bigg(\int\limits_{P_n[\mathbb{0},\textit{\textbf{t}}]} \displaystyle\prod_{i=0}^{m} \Lambda^{(i)}(b^i\textit{\textbf{s}})\bigg(\displaystyle\prod_{i=m+1}^{l} \Lambda^{(i)}(b^i\textit{\textbf{s}})-1\bigg) d\textit{\textbf{s}} \bigg)^{p}\]
\begin{equation}\label{ints_holder_lq_fast} =  \int\limits_{(P_n[\mathbb{0},\textit{\textbf{t}}])^{p}}\displaystyle\prod_{i=0}^{m}\bigg( E \prod_{j=1}^{p}\Lambda^{(i)}(b^i\textit{\textbf{s}}_j)\bigg) E\bigg(\prod_{j=1}^{p}\bigg( \displaystyle \prod_{i=m+1}^{l} \Lambda^{(i)}(b^i\textit{\textbf{s}}_j) - 1\bigg)\bigg)\prod_{j=1}^{p}d\textit{\textbf{s}}_j.\end{equation}  By applying H\"{o}lder's inequality with the conjugates $\gamma/(\gamma-1)$ and $\gamma>1,$ one gets that the integral in \eqref{ints_holder_lq_fast} is bounded by
\[\bigg(\int\limits_{(P_n[\mathbb{0},\textit{\textbf{t}}])^{p}}\displaystyle\prod_{i=0}^{m}\bigg( E \prod_{j=1}^{p}\Lambda^{(i)}(b^i\textit{\textbf{s}}_j)\bigg)^{\frac{\gamma}{\gamma-1}}\prod_{j=1}^{p}d\textit{\textbf{s}}_j\bigg)^{\frac{\gamma-1}{\gamma}}\]
\[\times \bigg(\int\limits_{(P_n[\mathbb{0},\textit{\textbf{t}}])^{p}} \bigg( E\prod_{j=1}^{p} \bigg( \displaystyle \prod_{i=m+1}^{l} \Lambda^{(i)}(b^i\textit{\textbf{s}}_j) - 1\bigg)\bigg)^{\gamma}\prod_{j=1}^{p}d\textit{\textbf{s}}_j\bigg)^{\frac{1}{\gamma}} := I_1^{\frac{\gamma-1}{\gamma}}\cdot I_2^{\frac{1}{\gamma}}. \]
As
\[ I_1 =\int\limits_{(P_n[\mathbb{0},\textit{\textbf{t}}])^{p}}\displaystyle\prod_{i=0}^{m}\rho^{\frac{\gamma}{\gamma-1}}(b^i\textit{\textbf{s}}_1, b^i\textit{\textbf{s}}_2,...,b^i\textit{\textbf{s}}_p) \prod_{j=1}^{p}d\textit{\textbf{s}}_j,\]
the integral $I_1$ is finite by the computations in the proof of Theorem \ref{th2} for the function $\rho^{\frac{\gamma}{\gamma-1}}(\cdot)$ instead of $\rho(\cdot).$ Assumption \ref{assump2} remains valid with $\textit{\textbf{p}}=(1,1,...,1)$ and the function $\rho^{\frac{\gamma}{\gamma-1}}(\cdot)$ as the power  $\frac{\gamma}{\gamma-1}>1.$ The condition \eqref{cond2.1} becomes \eqref{rate_cond2_lq_fast} and \eqref{cond2.2} holds as
\[\sum_{i=C}^\infty\ln\left(\rho(\mathbb{0}, b^i\mathbb{1}, 2b^i\mathbb{1},...,(k-1)b^i\mathbb{1}, \textit{\textbf{p}})\right) \leq \int\limits_{C-1}^\infty \ln \left(\rho(\mathbb{0}, b^u\mathbb{1}, 2b^u\mathbb{1},...,(k-1)b^u\mathbb{1}, \textit{\textbf{p}})\right)du\]
\[ \leq C_1 \int\limits_{C-1}^\infty \frac{1}{ b^{u\alpha}}du = C\int\limits_{C_2}^\infty\frac{dx}{x^{\alpha+1}}<\infty,\] where the changes of variables $u=\log_bx$ was used.

Now, let us consider the integral $I_2$. For any $x_j\in\mathbb{R},\ i_j = \{0,1\}, \ j=\overline{1,{p}},$ the next relationship holds true
\begin{equation}
\label{sums_fast} \prod_{j=1}^{p}(x_j-1) = \sum_{(i_1,i_2,..,i_{p})}\prod_{j=1}^{p}x_j^{i_j}(-1)^{1-i_j}
= \sum_{{\substack{(i_1,i_2,..,i_{p}) \\ i_1+...+i_{p} \in \overline{\mathbb{N}}_o}}}\prod_{j=1}^{p}x_j^{i_j} - \sum_{{\substack{(i_1,i_2,..,i_{p}) \\ i_1+...+i_{p} \in \mathbb{N}_o}}}\prod_{j=1}^{p}x_j^{i_j},
\end{equation} where $\mathbb{N}_o$ denotes the sets of integer odd numbers and  $\overline{\mathbb{N}}_o$ is its complement.

Applying $\eqref{sums_fast}$ to  $x_j = \prod_{i=m+1}^{l}\Lambda^{(i)}(b^i\textit{\textbf{s}}_j),$ one obtains
\[ E\prod_{j=1}^{p} \bigg( \displaystyle \prod_{i=m+1}^{l} \Lambda^{(i)}(b^i\textit{\textbf{s}}_j) - 1\bigg) \]
\[ = \sum_{{\substack{(i_1,i_2,..,i_{p}) \\ i_1+...+i_{p}\in \overline{\mathbb{N}}_o}}}\prod_{i=m+1}^{l}E\bigg( \prod_{j=1}^{p} \left(\Lambda^{(i)}(b^i\textit{\textbf{s}}_j) \right)^{i_j}\bigg)-\sum_{{\substack{(i_1,i_2,..,i_{p}) \\ i_1+...+i_{p} \in \mathbb{N}_o}}}\prod_{i=m+1}^{l}E \bigg(\prod_{j=1}^{p} \left(\Lambda^{(i)}(b^i\textit{\textbf{s}}_j) \right)^{i_j}\bigg)\]
\begin{equation}
\label{sum_even}
\leq  \sum_{{\substack{(i_1,i_2,..,i_{p}) \\ i_1+...+i_{p} \in \overline{\mathbb{N}}_o}}}\bigg( \prod_{i=m+1}^{l} E\bigg(  \prod_{j=1}^{p} \left(\Lambda^{(i)}(b^i\textit{\textbf{s}}_j) \right)^{i_j}\bigg)-1 \bigg),
\end{equation} as the both sums in the second expression consist of $2^{p-1}$ terms and each term in the second sum is bounded from below by 1.

By applying the H\"{o}lder's inequality to the last sum, one gets the upper bound
\[ \left(\sum_{{\substack{(i_1,i_2,..,i_{p}) \\ i_1+...+i_{p} \in \overline{\mathbb{N}}_o}}}1\right)^{\frac{\gamma-1}{\gamma}}\cdot \left(\sum_{{\substack{(i_1,i_2,..,i_{p}) \\ i_1+...+i_{p} \in \overline{\mathbb{N}}_o}}}\left( \prod_{i=m+1}^{l}E\left(  \prod_{j=1}^{p} \left(\Lambda^{(i)}(b^i\textit{\textbf{s}}_j) \right)^{i_j}\right)-1 \right)^{\gamma}\right)^{\frac{1}{\gamma}}\]
\[ \leq 2^{p-1} \left(\sum_{{\substack{(i_1,i_2,..,i_{p}) \\ i_1+...+i_{p} \in \overline{\mathbb{N}}_o}}}\left( \prod_{i=m+1}^{\infty}E\left(  \prod_{j=1}^{p} \left(\Lambda^{(i)}(b^i\textit{\textbf{s}}_j) \right)^{i_j}\right)-1 \right)^{\gamma}\right)^{\frac{1}{\gamma}},\] as $E\left(\prod_{j=1}^{p} \left(\Lambda^{(i)}(b^i\textit{\textbf{s}}_j) \right)^{i_j}\right) \geq 1,\ \textit{\textbf{s}}_1,...,\textit{\textbf{s}}_p\in\mathbb{R}^n, \ i\in \mathbb{N},$ by the condition $\rho(\cdot)\geq1.$

Thus, from \eqref{rate_lq_eq1_proof}, \eqref{ints_holder_lq_fast} and the above estimate it follows
\[ E|A(\textit{\textbf{t}})-A_m(\textit{\textbf{t}})|^{p} \]
\begin{equation} \label{sum_int_lq_fast} \leq C \left( \sum_{{\substack{(i_1,i_2,..,i_{p}) \\ i_1+...+i_{p} \in \overline{\mathbb{N}}_o}}} \int\limits_{(P_n[\mathbb{0},\textit{\textbf{t}}])^{p}}\left( \prod_{i=m+1}^{\infty}E\left(  \prod_{j=1}^{p} \left(\Lambda^{(i)}(b^i\textit{\textbf{s}}_j) \right)^{i_j}\right)-1 \right)^{\gamma} \prod_{j=1}^{p}d\textit{\textbf{s}}_j\right)^{\frac{1}{\gamma}}.\end{equation} According to Assumption {\rm \ref{assump2}}
\begin{equation}
\label{rate_lq_fast_mix}
E\left(\prod_{j=1}^{p} \left(\Lambda^{(i)}(b^i\textit{\textbf{s}}_j) \right)^{i_j}\right) \leq \rho\left(\mathbb{0}, \frac{\min\limits_{\substack{l,h: l\neq h}}||\textit{\textbf{s}}_l-\textit{\textbf{s}}_h||b^i\mathbb{1}}{\left(\sum\limits_{j=1}^pi_j-1\right)\sqrt{n}},...,\frac{\left({\sum\limits_{j=1}^pi_j}-1\right)\min\limits_{\substack{l,h: l\neq h}}||\textit{\textbf{s}}_l-\textit{\textbf{s}}_h||b^i\mathbb{1}}{\left({\sum\limits_{j=1}^pi_j}-1\right)\sqrt{n}}  \right)\end{equation}
and \eqref{sum_int_lq_fast} can be estimated from above as
\[ C \sum_{{\substack{(i_1,i_2,..,i_{p}) \\ i_1+...+i_{p} \in \overline{\mathbb{N}}_o}}} \sum\limits_{\substack{l,h: l\neq h}} \left(\int\limits_{(P_n[\mathbb{0},\textit{\textbf{t}}])^{p}} \left(\prod_{i=m+1}^{\infty}  \rho \left(\mathbb{0}, \frac{||\textit{\textbf{s}}_l-\textit{\textbf{s}}_h||b^i\mathbb{1}}{(\sum\limits_{j=1}^pi_j-1)\sqrt{n}},\right.\right.\right.\] \[\left.\left.\left. ...,\frac{({\sum\limits_{j=1}^pi_j}-1 ) ||\textit{\textbf{s}}_l-\textit{\textbf{s}}_h||b^i\mathbb{1}}{({\sum\limits_{j=1}^pi_j}-1 )\sqrt{n}}\right)  -1 \right)^{\gamma} \prod_{j=1}^{p}d\textit{\textbf{s}}_j \right)^{\frac{1}{\gamma}}.\]
The change of variables $\textit{\textbf{s}} = \textit{\textbf{s}}_l-\textit{\textbf{s}}_h$ and \eqref{sum_int_lq_fast} result in
\[ E|A(\textit{\textbf{t}})-A_m(\textit{\textbf{t}})|^{p} \leq C \left(\prod\limits_{i=1}^n t_i\right)^\frac{{{p}-1}}{\gamma} \] \[\times\left( \sum_{{\substack{(i_1,i_2,..,i_{p}) \\ i_1+...+i_{p} \in\overline{\mathbb{N}}_o}}}  \int\limits_{{\substack{P_n[\mathbb{0},\textit{\textbf{t}}]\circleddash \\ P_n[\mathbb{0},\textit{\textbf{t}}]}}} \left( \prod_{i=m+1}^{\infty}  \rho\left(\mathbb{0}, \frac{ ||\textit{\textbf{s}} ||b^i\mathbb{1}}{\big({\sum\limits_{j=1}^pi_j}-1\big)\sqrt{n}}, ...,\frac{\big({\sum\limits_{j=1}^pi_j}-1\big) ||\textit{\textbf{s}} ||b^i\mathbb{1}}{\big({\sum\limits_{j=1}^pi_j}-1\big)\sqrt{n}}  \right) -1 \right)^{\gamma} d\textit{\textbf{s}}\right)^{\frac{1}{\gamma}}.\] Denoting $\sum\limits_{j=1}^pi_j = q'$ and using the hyperspherical coordinates in the above integral, one gets
\[ \int\limits_{{\substack{P_n[\mathbb{0},\textit{\textbf{t}}]\circleddash \\ P_n[\mathbb{0},\textit{\textbf{t}}]}}} \bigg( \prod_{i=m+1}^{\infty}  \rho\bigg(\mathbb{0}, \frac{ ||\textit{\textbf{s}} ||b^i\mathbb{1}}{ (q'-1 )\sqrt{n}},...,\frac{ (q'-1 ) ||\textit{\textbf{s}} ||b^i\mathbb{1}}{ (q'-1 )\sqrt{n}}  \bigg) -1 \bigg)^{\gamma} d\textit{\textbf{s}} \]
\[ \leq C \int\limits_{0}^{\sqrt{n}} u^{n-1} \bigg( \prod_{i=m+1}^{\infty}  \rho\bigg(\mathbb{0}, \frac{ub^i\mathbb{1}}{ (q'-1 )\sqrt{n}},...,\frac{ (q'-1 ) u b^i\mathbb{1}}{ (q'-1 )\sqrt{n}}  \bigg) -1 \bigg)^{\gamma} du.\] The change of variables $t=\frac{b^m u}{(q'-1)\sqrt{n}}$ gives
\[ \frac{C}{b^{nm}} \int\limits_{0}^{b^m/(q'-1)} t^{n-1} \bigg( \prod_{i=1}^{\infty} \rho\bigg(\mathbb{0}, b^it\mathbb{1},2b^it\mathbb{1},...,({q'}-1)b^it\mathbb{1} \bigg) -1 \bigg)^\gamma dt.\] As $\rho(\cdot)$ is a nonincreasing function, it can be bounded from above as
\[ \frac{C}{b^{nm}} \int\limits_{0}^{\infty} t^{n-1} \bigg( \exp\bigg( \int\limits_{0}^{\infty} \ln(\rho(\mathbb{0}, b^xt\mathbb{1},2b^xt\mathbb{1},...,({q'}-1)b^xt\mathbb{1}) dx \bigg) -1 \bigg)^\gamma dt. \]  The change of variables $y = b^x t $ gives
\begin{equation}\label{int_int_rate_lq_fast}
\frac{C}{b^{nm}} \int_0^{\infty} t^{n-1}\bigg(\exp\bigg( \int_t^{\infty}\frac{\ln(\rho(\mathbb{0}, y \mathbb{1},2y \mathbb{1},...,({q'}-1) y \mathbb{1}) dx  }{y\ln(b)}dy\bigg) - 1\bigg)^\gamma dt.
\end{equation}

Let us show that the last integral is finite. When $t\to\infty,$
\[\int_t^{\infty}\frac{\ln(\rho(\mathbb{0}, y \mathbb{1},2y \mathbb{1},...,({q'}-1) y \mathbb{1}) }{y}dy\to0\] and, therefore, for sufficiently large $t$
\[ t^{n-1} \bigg( \exp\bigg(\int_t^{\infty}\frac{\ln(\rho(\mathbb{0}, y \mathbb{1},2y \mathbb{1},...,({q'}-1) y \mathbb{1})  }{y\ln(b)}dy \bigg) - 1 \bigg)^\gamma \]  \[\leq Ct^{n-1} \bigg(  \int_t^{\infty}\frac{\ln(\rho(\mathbb{0}, y \mathbb{1},2y \mathbb{1},...,({q'}-1) y \mathbb{1}) }{y}dy\bigg)^\gamma \leq Ct^{n-1}\bigg(\int_t^\infty\frac{dy}{y^{1+\alpha}}\bigg)^\gamma \leq Ct^{n-1-\gamma\alpha}.\]  As $\gamma \in (\max(1, n/\alpha), n\ln(b)/\ln({E\Lambda^p\mathbb{0}))},$ then $n-1-\gamma\alpha<-1,$ and the integrand in \eqref{int_int_rate_lq_fast} is integrable on $[C_2,\infty),$ $C_2>0.$

Now, considering  $t\to 0,$ one obtains
\[  t^{n-1} \bigg( \exp\bigg(\int_t^{\infty}\frac{\ln(\rho(\mathbb{0}, y \mathbb{1},2y \mathbb{1},...,({q'}-1) y \mathbb{1})  }{y\ln(b)}dy \bigg) -1 \bigg)^\gamma \]
\[\leq  t^{n-1} \bigg( \exp\bigg(\int_t^{C}\frac{\ln(\rho(\mathbb{0}, y \mathbb{1},2y \mathbb{1},...,({q'}-1) y \mathbb{1})  }{y\ln(b)}dy  + \int_C^{\infty}\frac{\ln(\rho(\mathbb{0}, y \mathbb{1},2y \mathbb{1},...,({q'}-1) y \mathbb{1})  }{y\ln(b)}dy \bigg)\bigg)^\gamma .\] The second integral is finite, so, for $t\to0$ the exponent is bounded from above by
\[  t^{n-1} \bigg(  \exp\bigg( \int_t^{C}\frac{\ln(\rho(\mathbb{0}, y \mathbb{1},2y \mathbb{1},...,({q'}-1)b^y \mathbb{1})  }{y\ln(b)}dy  + C_2 \bigg)\bigg)^\gamma \]
\[ \leq  t^{n-1} \bigg( \exp\bigg( C_2-\ln t\cdot\frac{\ln E\Lambda^{q'}(\mathbb{0})}{\ln b} \bigg)\bigg)^\gamma=C\cdot t^{n-1-\gamma \frac{\ln(E\Lambda^{q'}(\mathbb{0}))}{\ln b}}.\] By theorem's assumption, $n-1-\gamma\frac{\ln E\Lambda^{q'}(\mathbb{0})}{\ln b}>-1.$ Hence, the function in \eqref{int_int_rate_lq_fast} is integrable on $[0,C_2].$

Thus,
\begin{equation}\label{rate_lq_even_fast}E|A(\textit{\textbf{t}})-A_m(\textit{\textbf{t}})|^p\leq C  \bigg(\frac{ \prod_{i=1}^nt_i ^{p-1}}{b^{nm}}\bigg)^{1/\gamma}.\end{equation}

Now, let us consider the case when $q$ is not even. Let $p_0\leq p$ is an  even integer such that $q = (1-\theta)p_0+\theta p,\ \theta\in(0,1).$
 By using Lyapunov's inequality $||f||_q^q\leq ||f||_{p_0}^{p_0(1-\theta)}||f||_{p}^{p\theta},$ we obtain
\[ E|A(\textit{\textbf{t}})-A_m(\textit{\textbf{t}})|^{q} \leq \bigg( E|A(\textit{\textbf{t}})-A_m(\textit{\textbf{t}})|^{p_0} \bigg)^{1-\theta} \bigg( E|A(\textit{\textbf{t}})-A_m(\textit{\textbf{t}})|^{p} \bigg)^{\theta}.\] As $p_0$ and $p$ are even, by applying \eqref{rate_lq_even_fast}
\[ E|A(\textit{\textbf{t}})-A_m(\textit{\textbf{t}})|^{q} \leq \bigg(C\frac{ \prod_{i=1}^nt_i^{p_0-1}}{b^{nm}}\bigg)^{\frac{1-\theta}{\gamma}}\bigg(C\frac{\prod_{i=1}^nt_i^{p -1}}{b^{nm}}\bigg)^{\frac{ \theta}{\gamma}} \leq C  \bigg(\frac{\prod_{i=1}^nt_i^{q-1}}{b^{nm}}\bigg)^{1/\gamma},\] which finishes the proof. \end{proof}

\begin{lemma}
\label{int_momnts}
Let Assumption  {\rm \ref{assump_main}} be satisfied, then the following inequality holds for all $l\geq m+1,\ m\in\mathbb{N},$
\[ \int\limits_{(P_n[\mathbb{0},\textit{\textbf{t}}])^q} E\bigg( \displaystyle\prod_{j=1}^{q-1}\prod_{i=m+1}^l \Lambda^{(i)} (b^i \textit{\textbf{s}}_j) \bigg) \displaystyle\prod_{j=1}^{q} d \textit{\textbf{s}}_j \leq \int\limits_{(P_n[\mathbb{0},\textit{\textbf{t}}])^q} E\bigg( \displaystyle\prod_{j=1}^{q}\prod_{i=m+1}^l \Lambda^{(i)} (b^i \textit{\textbf{s}}_j) \bigg) \displaystyle\prod_{j=1}^{q} d \textit{\textbf{s}}_j.\]
\end{lemma}

\begin{proof}
Let us consider the representation
\[ \int\limits_{(P_n[\mathbb{0},\textit{\textbf{t}}])^q} E\bigg( \displaystyle\prod_{j=1}^{q-1}\prod_{i=m+1}^l \Lambda^{(i)} (b^i \textit{\textbf{s}}_j) \bigg) \displaystyle\prod_{j=1}^{q} d \textit{\textbf{s}}_j= \displaystyle\prod_{j=1}^{n}t_j \int\limits_{(P_n[\mathbb{0},\textit{\textbf{t}}])^{q-1}} E\bigg( \displaystyle\prod_{j=1}^{q-1}\prod_{i=m+1}^l \Lambda^{(i)} (b^i \textit{\textbf{s}}_j) \bigg) \displaystyle\prod_{j=1}^{q-1} d \textit{\textbf{s}}_j \]
\[ = \displaystyle\prod_{j=1}^{n}t_j  E\bigg( \int\limits_{P_n[\mathbb{0},\textit{\textbf{t}}]} \displaystyle\prod_{i=m+1}^{l}  \Lambda^{(i)} (b^i \textit{\textbf{s}})d \textit{\textbf{s}} \bigg)^{q-1}.\] Then, by applying Jensen's inequality twice, one obtains its upper bound
\[ \displaystyle\prod_{j=1}^{n}t_j  \bigg( E\bigg( \int\limits_{P_n[\mathbb{0},\textit{\textbf{t}}]} \displaystyle\prod_{i=m+1}^{l}  \Lambda^{(i)} (b^i \textit{\textbf{s}})d \textit{\textbf{s}} \bigg)^{q}\bigg)^{\frac{q-1}{q}} = \frac{\displaystyle\prod_{j=1}^{n}t_j E\bigg( \int\limits_{P_n[\mathbb{0},\textit{\textbf{t}}]} \displaystyle\prod_{i=m+1}^{l}  \Lambda^{(i)} (b^i \textit{\textbf{s}})d \textit{\textbf{s}} \bigg)^{q} }{\bigg( E\bigg( \displaystyle\int\limits_{P_n[\mathbb{0},\textit{\textbf{t}}]} \displaystyle\prod_{i=m+1}^{l}  \Lambda^{(i)} (b^i \textit{\textbf{s}})d \textit{\textbf{s}} \bigg)^{q}\bigg)^{1/q}}\]
\[ \leq \frac{\displaystyle\prod_{j=1}^{n}t_jE\bigg( \int\limits_{P_n[\mathbb{0},\textit{\textbf{t}}]} \displaystyle\prod_{i=m+1}^{l}  \Lambda^{(i)} (b^i \textit{\textbf{s}})d \textit{\textbf{s}} \bigg)^{q} }{ E\bigg( \displaystyle\int\limits_{P_n[\mathbb{0},\textit{\textbf{t}}]} \displaystyle\prod_{i=m+1}^{l}  \Lambda^{(i)} (b^i \textit{\textbf{s}})d \textit{\textbf{s}} \bigg)} = \int\limits_{(P_n[\mathbb{0},\textit{\textbf{t}}])^q} E\bigg( \displaystyle\prod_{j=1}^{q}\prod_{i=m+1}^l \Lambda^{(i)} (b^i \textit{\textbf{s}}_j) \bigg) \displaystyle\prod_{j=1}^{q} d \textit{\textbf{s}}_j. \] \end{proof}

\begin{corollary}
\label{rate_theorem_q}
Let Assumptions {\rm \ref{assump_main}} and {\rm \ref{assump2}} hold  true and $\rho(\cdot)\geq 1$ for the vector $\textit{\textbf{p}} = (1,1,...,1)$ that has $ {p}$ components, where $p $ is an even integer. Also, let there exist such $x_0$ and $C_1>0$ that for all $x\geq x_0$
\begin{equation*}
\ln(\rho(\mathbb{0},x\mathbb{1},...,(p-1)x\mathbb{1} )) \leq C_1x^{-\alpha}, \ \alpha> n.
\end{equation*}
If
\begin{equation*}
b > \left(E\Lambda^{p}(\mathbb{0})\right)^{1/n},
\end{equation*}
then for all $\textit{\textbf{t}}\in P_n[\mathbb{0},\mathbb{1}]$ and $q\in[2,p]$
\begin{equation*}
E|A(\textit{\textbf{t}})-A_m(\textit{\textbf{t}})|^q\leq C \bigg(\prod\limits_{i=1}^nt_i\bigg)^{q-1}\left(\frac{E\Lambda^{p}(\mathbb{0})}{b^n}\right)^m.
\end{equation*}
\end{corollary}

\begin{proof} The main steps of the proof are similar to the proof of Theorem \ref{rate_theorem_lq_fast}. By applying H\"{o}lder's inequality $||fg||_1\leq||f||_\infty||g||_1 $ in \eqref{ints_holder_lq_fast}, one gets
\[ E|A_l(\textit{\textbf{t}})-A_m(\textit{\textbf{t}})|^{p}\leq \prod_{i=0}^{m} \esssup\limits_{\textit{\textbf{s}}_j\in P_n[\mathbb{0},\textit{\textbf{t}}] ,\ j=\overline{1,p}} \bigg\{ E \prod_{j=1}^{p}\Lambda^{(i)}(b^i\textit{\textbf{s}}_j) \bigg\}\]
\begin{equation}\label{i_1} \times \int\limits_{(P_n[\mathbb{0},\textit{\textbf{t}}])^{p}}  E\bigg(\prod_{j=1}^{p}\bigg( \displaystyle \prod_{i=m+1}^{l} \Lambda^{(i)}(b^i\textit{\textbf{s}}_j) - 1\bigg)\bigg)\prod_{j=1}^{p}d\textit{\textbf{s}}_j \end{equation}
\begin{equation}
\label{higher_moments_int}
\leq  (E\Lambda^{p}(\mathbb{0}))^m \int\limits_{(P_n[\mathbb{0}, \textit{\textbf{t}}])^{p}}E\bigg(\prod_{j=1}^{p}\bigg( \displaystyle \prod_{i=m+1}^{l} \Lambda^{(i)}(b^i\textit{\textbf{s}}_j) - 1\bigg)\bigg)\prod_{j=1}^{p}d\textit{\textbf{s}}_j.
\end{equation}  Then, analogously to \eqref{sum_even} the expression in \eqref{higher_moments_int} is bounded by
\[ (E\Lambda^{p}(\mathbb{0}))^m \sum_{{\substack{(i_1,i_2,..,i_{p}) \\ i_1+...+i_{p} \in \overline{\mathbb{N}}_o}}} \int\limits_{(P_n[\mathbb{0}, \textit{\textbf{t}}])^{p}} \bigg( \prod_{i=m+1}^{l}E\bigg(  \prod_{j=1}^{p}\left(\Lambda^{(i)}(b^i\textit{\textbf{s}}_j) \right)^{i_j}\bigg)-1 \bigg) \prod_{j=1}^{p}d\textit{\textbf{s}}_j.\] By applying Lemma \ref{int_momnts} to the above expression, and as $\rho(\cdot)\geq 1,$ one can see that the latter is majorized by
\[ 2^{{p}-1} (E\Lambda^{p}(\mathbb{0}))^m \int\limits_{(P_n[\mathbb{0}, \textit{\textbf{t}}])^{p}} \bigg( \prod_{i=m+1}^{l}E\bigg(  \prod_{j=1}^{p} \Lambda^{(i)}(b^i\textit{\textbf{s}}_j)  \bigg)-1 \bigg) \prod_{j=1}^{p}d\textit{\textbf{s}}_j\]
\begin{equation}\label{rho_tilde} \leq 2^{{p}-1} (E\Lambda^{p}(\mathbb{0}))^m \int\limits_{(P_n[\mathbb{0}, \textit{\textbf{t}}])^{p}} \bigg( \prod_{i=m+1}^{\infty}\rho(b^i\textit{\textbf{s}}_1,b^i\textit{\textbf{s}}_2,...,b^i\textit{\textbf{s}}_{p})-1 \bigg) \prod_{j=1}^{p}d\textit{\textbf{s}}_j.\end{equation} Hence,
\[ E|A(\textit{\textbf{t}})-A_m(\textit{\textbf{t}})|^{p} \leq C (E\Lambda^{p}(\mathbb{0}))^m \int\limits_{(P_n[\mathbb{0}, \textit{\textbf{t}}])^{p}} \bigg( \prod_{i=m+1}^{\infty}\rho(b^i\textit{\textbf{s}}_1,b^i\textit{\textbf{s}}_2,...,b^i\textit{\textbf{s}}_{p})-1 \bigg) \prod_{j=1}^{p}d\textit{\textbf{s}}_j.\]
Now, using the inequality \eqref{rate_lq_fast_mix} for $\sum_{j=1}^pi_j=p,$ i.e.
\[\rho(b^i\textit{\textbf{s}}_1,b^i\textit{\textbf{s}}_2,...,b^i\textit{\textbf{s}}_{p}) \] \[\leq \rho\bigg(\mathbb{0}, \frac{b^i\min\limits_{\substack{l,h: l\neq h}}||\textit{\textbf{s}}_l-\textit{\textbf{s}}_h||\mathbb{1}}{({p}-1)\sqrt{n}},\frac{2b^i\min\limits_{\substack{l,h: l\neq h}}||\textit{\textbf{s}}_l-\textit{\textbf{s}}_h||\mathbb{1}}{({p}-1)\sqrt{n}},...,\frac{(p-1)b^i\min\limits_{\substack{l,h: p\neq h}}||\textit{\textbf{s}}_l-\textit{\textbf{s}}_h||\mathbb{1}}{({p}-1)\sqrt{n}}  \bigg)\] and $\gamma=1,$ as in the proof of Theorem \ref{rate_theorem_lq_fast}, we obtain that

\begin{equation*} E|A(\textit{\textbf{t}})-A_m(\textit{\textbf{t}})|^{p} \leq C\bigg(\prod\limits_{i=1}^nt_i\bigg)^{{p}-1}\bigg(\frac{E\Lambda^{p}(\mathbb{0})}{b^n}\bigg)^m.\end{equation*}
The result for $q\in[2,p]$ follows by Lyapunov's inequality, which finishes the proof. \end{proof}

The obtained results can also be reformulated in terms of the function $\widetilde{\rho}(\cdot).$ Namely, the following statement holds true.

\begin{corollary} \label{cor_widetilde_rho}
Let Assumptions {\rm \ref{assump_main}} and {\rm \ref{assump2'}} be satisfied, and $\widetilde{\rho}(\cdot)\geq 1$ for the vector $\textit{\textbf{p}} = (1,1,...,1)$ that has $ {p}$ components, where $p$ is an even integer. Also, let there exist such $x_0$ and $C_1>0$ that for all $x\geq x_0$

\begin{equation*}
\ln(\widetilde{\rho}(\mathbb{0},x\mathbb{1},...,(p-1)x\mathbb{1} )) \leq C_1x^{-\alpha}, \ \alpha>n.
\end{equation*}
If
\begin{equation*}
b > \left(\widetilde{\rho}(\mathbb{0},\mathbb{0},...,\mathbb{0})\right)^{1/n},
\end{equation*}
then for all $\textit{\textbf{t}}\in P_n[\mathbb{0},\mathbb{1}]$ and $q\in[2,p]$
\begin{equation*}
E|A(\textit{\textbf{t}})-A_m(\textit{\textbf{t}})|^q\leq C \bigg(\prod\limits_{i=1}^nt_i\bigg)^{q-1}\left(\frac{\widetilde{\rho}(\mathbb{0},\mathbb{0},...,\mathbb{0})}{b^n}\right)^m.
\end{equation*}

\end{corollary}

\begin{proof}
The proof is analogous to the proof of Corollary {\rm  \ref{rate_theorem_q}}. Two modifications of the estimates in {\rm \eqref{i_1}} and {\rm \eqref{rho_tilde}} are required. Namely, the estimates should be written in terms of the \mbox{function $\widetilde{\rho}(\cdot)$} by using the inequalities
\[ \esssup\limits_{\textit{\textbf{s}}_j\in P_n[\mathbb{0},\textit{\textbf{t}}] ,\ j=\overline{1,p}} \bigg\{ E \prod_{j=1}^{p}\Lambda^{(i)}(b^i\textit{\textbf{s}}_j) \bigg\} \leq  \widetilde{\rho}(\mathbb{0},\mathbb{0},...,\mathbb{0}), \]  and
\[ E\bigg(  \prod_{j=1}^{p} \Lambda^{(i)}(b^i\textit{\textbf{s}}_j)  \bigg) \leq \widetilde{\rho}(b^i\textit{\textbf{s}}_1,b^i\textit{\textbf{s}}_2,...,b^i\textit{\textbf{s}}_{p}).\]\vspace{-1.2cm}

\
{}
\end{proof}

\begin{corollary}
For $\gamma<\frac{n}{n-\log_bE\Lambda^p(\mathbb{0})}$ the rate of convergence in Theorem {\rm\ref{rate_theorem_lq_fast}} is faster than in Corollary~{\rm\ref{rate_theorem_q}}, and vice versa for $\gamma>\frac{n}{n-\log_bE\Lambda^p(\mathbb{0})}.$
\end{corollary}

\begin{proof}
To compare the rates of convergences, one needs to compare the terms $b^{-\frac{nm}{\gamma}}$ and $\frac{(E\Lambda^p(\mathbb{0}))^m}{b^{nm}},$ as $m\to\infty.$ Let us rewrite the second rate as
\[ \frac{(E\Lambda^p(\mathbb{0}))^m}{b^{nm}} = b^{m\log_bE\Lambda^p(\mathbb{0})- nm} . \]
Thus, $b^{-\frac{nm}{\gamma}} < \frac{(E\Lambda^p(\mathbb{0}))^m}{b^{nm}},$ as $m\to\infty,$ if $\frac{nm}{\gamma}>nm- m\log_bE\Lambda^p(\mathbb{0}),$ which is equivalent to $\gamma<\frac{n}{n-\log_bE\Lambda^p(\mathbb{0})}.$ On the other hand, $b^{-\frac{nm}{\gamma}} > \frac{(E\Lambda^p(\mathbb{0}))^m}{b^{nm}},$ if $\gamma>\frac{n}{n-\log_bE\Lambda^p(\mathbb{0})}.$\end{proof}

Using analogous calculations one obtains rates of convergence for the random measures $\mu_m(\cdot)$ in the spaces $L_q.$
\begin{corollary}
Let the conditions of Theorem {\rm \ref{rate_theorem_lq_fast}} or Corollary {\rm \ref{rate_theorem_q}} be satisfied. Then, for the random measure and Borel subsets  $B \in \mathfrak{B}$ defined in Corollary {\rm \ref{measure}},  the corresponding rates of convergence hold true
\[E|\mu(B)-\mu_m(B)|^q \leq  C  \bigg(\frac{ |B| ^{q-1}}{b^{nm}}\bigg)^{1/\gamma}\] or
\[ E|\mu(B)-\mu_m(B)|^q \leq C |B|^{q-1}\left(\frac{E\Lambda^{p}(\mathbb{0})}{b^n}\right)^m.\]
\end{corollary}

\section{Scaling of moments and the R{\'e}nyi function.}
\label{sec3}
This section obtains estimates of moments of $A(\cdot).$ Then, these estimates are used to calculate the R{\'e}nyi function $T(\cdot)$ of the measure $\mu(\cdot)$.

The R{\'e}nyi function of the random measure $\mu(\cdot)$ is defined via moments $E\left(\mu\left(B^{(k)}_l\right)\right)^q,$ where $B^{(k)}_l$ are the hypercubes forming the dyadic decompositions of $P_n[\mathbb{0}, \mathbb{1}].$ It follows from the definition of the measures $\mu_m(\cdot)$ that the limiting measure $\mu(\cdot)$ is homogeneous, i.e. $\mu\left(B^{(j)}_i\right) \overset{d}{=} \mu\left(B^{(j)}_l\right),\ i,l=0,1,...,2^{nj}-1, \ j=1,2,...$.  Hence, to derive the R{\'e}nyi function of $\mu(\cdot)$ it is enough to study the moments $E\left(\mu(P_n[\mathbb{0}, \textit{\textbf{t}}])\right)^q,$ where $\textit{\textbf{t}} = (t,t,...,t), \ t\in[0,1].$

In the following we assume that $\Lambda(\cdot)$ is strictly homogeneous, i.e. its finite-fimensional distributions are invariant with respect to shifts.

\begin{lemma}
\label{lemma_moments}
Let the conditions of Theorem {\rm \ref{th2}} hold true, i.e. $A(\textit{\textbf{t}})\in L_q,\ q\geq 0,$ $\textit{\textbf{t}} \in P_n[\mathbb{0},\mathbb{1}].$
If for $ q\in(0,1)$ the function $\rho(\mathbb{0}, \textit{\textbf{x}}, \textit{\textbf{q}}), \ \textit{\textbf{q}}=(q-1,1),$ is nondecreasing in $||\textit{\textbf{x}}||,$ satisfies
\[\sum_{i=1}^{\infty} \ln\left(\frac{\rho(\mathbb{0}, b^{-i}\mathbb{1},\textit{\textbf{q}})}{E\Lambda^q(\mathbb{0})}\right)<\infty,\]  and for $q\geq1$ and some $k$-dimensional vector $\widetilde{\textit{\textbf{p}}}=({q}/{k},..,{q}/{k})$ it holds
\[ \sum_{i=1}^{\infty} \ln\left(\frac{E\Lambda^q(\mathbb{0})}{\rho(\mathbb{0}, b^{-i}\mathbb{1},..., b^{-i}(k-1)\mathbb{1},\widetilde{\textit{\textbf{p}}})}\right) < \infty,\] then, there exist constants $C_1 ,\ C_2$ such that
\begin{equation}
\label{renyi} C_1 t^{n q- \log_bE\Lambda^q(\mathbb{0})} \leq  EA^q(t\mathbb{1}) \leq C_2 t^{n q- \log_bE\Lambda^q(\mathbb{0})},\  t\to0,
\end{equation} and the random variables $A(t\mathbb{1})$ are nondegenerate, that is $P(A(t\mathbb{1})>0)>0.$
\end{lemma}

\begin{proof}

Martingale properties of $A_m(t\mathbb{1})$ will be used to prove \eqref{renyi}. Two different cases  will be considered as $A^q_m(t\mathbb{1})$ is a submartingale if $q \geq 1$, and it is a supermartingale if $q < 1.$ To obtain the upper bound when $q\geq 1,$ we will derive a  uniform in $m$ estimate from above for $EA^q_m(t\mathbb{1})$. As $A_m(t\mathbb{1})$ converges to $A(t\mathbb{1})$ in the space $L_q$, the same estimate will hold for $EA^q(t\mathbb{1})$.

Let $q\geq 1$ and $m_\textit{\textbf{t}} = [-\log_bt]$ be the largest integer such that $m_\textit{\textbf{t}} \leq -\log_bt.$  Then, by H{\"o}lder's inequality with $1/q+1/p=1$ it follows
\[A_{m+1}^q(t\mathbb{1}) = \left( \int\limits_{P_n[\mathbb{0}, t \mathbb{1}]} \prod_{l=0}^{m_{\textit{\textbf{t}}}-1} \Lambda^{(l)}(b^l\textit{\textbf{u}}) \prod_{l=m_{\textit{\textbf{t}}}}^{m} \Lambda^{(l)}(b^l\textit{\textbf{u}} )d\textit{\textbf{u}} \right)^q \]
\[\leq \left( \int\limits_{P_n[\mathbb{0},t \mathbb{1}]} \left( \prod_{l=0}^{m_{\textit{\textbf{t}}}-1} (\Lambda^{(l)}(b^l\textit{\textbf{u}} ))\right)^q \prod_{l = m_{\textit{\textbf{t}}}}^{m}\Lambda^{(l)}(b^l\textit{\textbf{u}} )  d\textit{\textbf{u}}\right) \left(\int\limits_{P_n[\mathbb{0},t\mathbb{1}]} \prod_{l = m_{\mathfrak{t}}}^{m}\Lambda^{(l)}(b^l\textit{\textbf{u}} )  d\textit{\textbf{u}}\right)^{q/p}.\]
Applying expectations to the both sides we obtain that the expectation $EA^q_{m+1}( t\mathbb{1})$ is bounded by \[\int\limits_{P_n[\mathbb{0},t\mathbb{1}]}\prod_{l=0}^{m_{t}-1}E (\Lambda^{(l)}(b^l\textit{\textbf{u}} ))^q  E\left( \left(\prod_{l = m_{ t}}^{m} \Lambda^{(l)}(b^l\textit{\textbf{u}} ) \right) \left( \int\limits_{P_n[\mathbb{0},t\mathbb{1}]} \prod_{l = m_{ t}}^{m} \Lambda^{(l)}(b^l\textit{\textbf{v}} ) d\textit{\textbf{v}} \right)^{q/p}  \right) d\textit{\textbf{u}}.\] Therefore,
\[ EA^q_{m+1}( t\mathbb{1}) \leq (E\Lambda^q(\mathbb{0}))^{m_t}E\left( \int\limits_{P_n[\mathbb{0},t\mathbb{1}]} \prod_{l = m_{t}}^{m} \Lambda^{(l)}(b^l\textit{\textbf{u}}) d\textit{\textbf{u}} \right)^{1+q/p} \]
\[= (E\Lambda^q(\mathbb{0}))^{m_t}E\left( \int\limits_{P_n[\mathbb{0},t\mathbb{1}]} \prod_{l = 0}^{m-m_{t}} \Lambda^{(l)}(b^{l+m_t}\textit{\textbf{u}}) d\textit{\textbf{u}} \right)^{q} \]
\begin{equation}\label{moments_estimation_upper}=  (E\Lambda^q(\mathbb{0}))^{m_t} \left(b^{-m_t }\right)^{nq} E A_{m-m_t+1}^q(b^{\log_bt + m_{\textit{\textbf{t}}}}\mathbb{1}).\end{equation}
As $(E\Lambda^q(\mathbb{0}))^{m_t} \leq (E\Lambda^q(\mathbb{0}))^{-\log_b t} = t^{-\log_bE\Lambda^q(\mathbb{0})},$ $b^{\log_bt+m_t}\leq 1,$ and
\[ b^{-m_tnq} = b^{(-\log_b t - [-\log_b t] + \log_b t)nq} \leq b^{(1+\log_b t)nq} = b^{nq} t^{nq},\] therefore,
\[ EA^q( t\mathbb{1}) \leq t^{nq-\log_bE\Lambda^q(\mathbb{0})} b^{nq} \sup_{\textit{\textbf{u}}\in P_n[\mathbb{0},\mathbb{1}]}EA^q(\textit{\textbf{u}}).\]

To find the upper bound for $q\in(0,1),$ we will use a recursive estimation for $EA_m^q(t\mathbb{1}).$ From H\"{o}lder's inequality $(EX^q)^{1/q} (EY^p)^{1/p}\leq E(XY)$ for $q\in(0,1), \ p<0,$ it follows that
\[ EX^q \leq \left( \frac{E(XY)}{(EY^p)^{1/p}}\right)^q = (E(XY))^{q} (EY^p)^{1-q}.\] By setting $X=A_{m+1}(t\mathbb{1})$ and $Y=(A_m^{1-q}(t\mathbb{1})(\Lambda^{(m)}(\mathbb{0}))^{1-q})^{-1},$ one obtains
\[ EA_{m+1}^q(t\mathbb{1}) \leq \left( E\left( \frac{A_{m+1}(t\mathbb{1})(\Lambda^{(m)}(\mathbb{0}))^{q-1}}{A_m^{1-q}(t\mathbb{1})} \right) \right)^q\left( EA_{m}^{(q-1)p}(t\mathbb{1}) (\Lambda^{(m)}(\mathbb{0}))^{(q-1)p} \right)^{1-q}.\] Let us consider the first expectation separately
\[E\left(  \frac{A_{m+1}(t\mathbb{1})(\Lambda^{(m)}(\mathbb{0}))^{q-1}}{A_m^{1-q}(t\mathbb{1})} \right) = E\left( \frac{\displaystyle\int_{P_n[\mathbb{0},t\mathbb{1}]}\prod_{i=0}^{m-1} \Lambda^{(i)}(b^i\textit{\textbf{u}})\Lambda^{(m)}(b^m\textit{\textbf{u}})(\Lambda^{(m)}(\mathbb{0}))^{q-1}d\textit{\textbf{u}}}{A_m^{1-q}(t\mathbb{1})} \right)\]
\[ \leq E  A^q_{m}(t\mathbb{1}) \max_{\textit{\textbf{u}} \in P_n[\mathbb{0},t\mathbb{1}] } E(\Lambda^{(m)}(b^m\textit{\textbf{u}})(\Lambda^{(m)}(\mathbb{0}))^{q-1}) \leq E  A^q_{m}(t\mathbb{1}) \rho(\mathbb{0}, b^mt\mathbb{1},\textit{\textbf{q}})\] \[\leq E  A^q_{m}(t\mathbb{1}) \rho(\mathbb{0}, b^{m-m_t }\mathbb{1},\textit{\textbf{q}}),\] where $\textit{\textbf{q}}=(q-1,1).$ The property that $\rho(\mathbb{0}, \textit{\textbf{x}},\textit{\textbf{q}})$ is a nondecreasing function of $||\textit{\textbf{x}}||$ for $q<1$ was used.

Thus,  from the above estimate and $(q-1)p=q$ it follows
\[ EA_{m+1}^q(t\mathbb{1}) \leq (EA^q_{m}(t\mathbb{1}) \rho(\mathbb{0}, b^{m-m_t }\mathbb{1},\textit{\textbf{q}}))^{q} \left( EA_{m}^{(q-1)p}(t\mathbb{1}) E(\Lambda^{(m)}(\mathbb{0}))^{(q-1)p} \right)^{1-q} \]
\[ = EA_{m}^q(t\mathbb{1}) E(\Lambda (\mathbb{0}))^q\left( \frac{\rho(\mathbb{0}, b^{m-m_t}\mathbb{1},\textit{\textbf{q}})}{E(\Lambda (\mathbb{0}))^q} \right)^q.\] By applying this estimate recursively, one gets
\[ EA_{m_t+1}^q(t\mathbb{1}) \leq  EA_{1}^q(t\mathbb{1}) (E(\Lambda(\mathbb{0}))^q)^{m_t}\prod\limits_{i=1}^{m_t} \left( \frac{\rho(\mathbb{0}, b^{i-m_t}\mathbb{1},\textit{\textbf{q}})}{E(\Lambda (\mathbb{0}))^q} \right)^q.\]
As $EA_{1}^q(t\mathbb{1})\leq(EA_{1}(t\mathbb{1}))^q=t^{nq},$ $(E(\Lambda(\mathbb{0}))^q)^{m_t} \leq (E(\Lambda(\mathbb{0}))^q)^{-\log_bt} = t^{-\log_bE\Lambda^q(\mathbb{0})},$
\[ EA^q_{m_t+1}(t\mathbb{1}) \leq t^{nq-\log_bE\Lambda^q(\mathbb{0})} \prod\limits_{i=1}^{m_t} \left( \frac{\rho(\mathbb{0}, b^{-i}\mathbb{1},\textit{\textbf{q}})}{E(\Lambda (\mathbb{0}))^q} \right)^q.\] It follows from $\sum_{i=1}^\infty\ln\left(\frac{\rho(\mathbb{0}, b^{-i}\mathbb{1},\textit{\textbf{q}})}{E(\Lambda (\mathbb{0}))^q)}\right)<\infty$ that $\prod\limits_{i=1}^{m_t} \left( \frac{\rho(\mathbb{0}, b^{-i}\mathbb{1},\textit{\textbf{q}})}{E(\Lambda (\mathbb{0}))^q} \right)^q<C$ and

\[ EA_{m_t+1}^q(t\mathbb{1}) \leq C t^{nq-\log_bE\Lambda^q(\mathbb{0})}.\] As $A_{m}^q(t\mathbb{1})$ is a supermartingale, $EA^q(t\mathbb{1}) \leq EA_{m_t+1}^q(t\mathbb{1}),$ which provides the required estimate from the above inequality.

Now, let us obtain the estimate from below for $q \geq 1.$ Notice, that by H\"{o}lder's inequality $||fg||_1 \geq ||f||_{\widetilde{p}}||g||_{\widetilde{q}}, \ \widetilde{p}>0, \ \widetilde{q}<0,$ it follows

\begin{equation}\label{holder_upper} ||fg||_1^q = ||f|g|^{1/\widetilde{p}}|g|^{1/\widetilde{q}}||_1^q \geq ||f|g|^{1/\widetilde{p}}||^{q}_{\widetilde{p}} |||g|^{1/\widetilde{q}}||^q_{\widetilde{q}}=|||f|^{\widetilde{p}}g||^{q/\widetilde{p}}_1||g||^{q/\widetilde{q}}_1.\end{equation}  Let us also choose $\widetilde{p}$ such that $q/\widetilde{p}=k\in\mathbb{N}$ and apply \eqref{holder_upper} to $f=\Lambda^{(m)}(\textit{\textbf{u}})$ and $g=\Lambda_m(\textit{\textbf{u}}):$
 \[ E A^q_{m+1}(t\mathbb{1}) = E\left( \int\limits_{P_n[\mathbb{0},t\mathbb{1}]} \Lambda_m(\textit{\textbf{u}}) \Lambda^{(m)}(b^m\textit{\textbf{u}}) d \textit{\textbf{u}}\right)^q \]
\[ \geq E\left( \int\limits_{P_n[\mathbb{0},t\mathbb{1}]} (\Lambda^{(m)}(\textit{\textbf{u}}))^{\widetilde{p}} \Lambda_m(\textit{\textbf{u}}) d\textit{\textbf{u}} \right)^{q/\widetilde{p}} \left( \int\limits_{P_n[\mathbb{0},t\mathbb{1}]} \Lambda_m(\textit{\textbf{u}}) d\textit{\textbf{u}} \right)^{q/\widetilde{q}}\]
\[ = \int\limits_{\left(P_n[\mathbb{0},t\mathbb{1}]\right)^k} E \prod\limits_{i=1}^k\left( \Lambda^{(m)} (b^m \textit{\textbf{u}}_i) \right)^{\widetilde{p}} E \left(\prod\limits_{i=1}^k \Lambda_{m}(\textit{\textbf{u}}_i) \left( \int\limits_{P_n[\mathbb{0},t\mathbb{1}]} \Lambda_m(\textit{\textbf{v}}) d\textit{\textbf{v}} \right)^{q/\widetilde{q}} \right) \prod\limits_{i=1}^k d \textit{\textbf{u}}_i \]
\[\geq  \min_{\substack{\textit{\textbf{u}}_i\in P_n[\mathbb{0},t\mathbb{1}], \\ i=\overline{1,k}}} \rho(b^m\textit{\textbf{u}}_1,...,b^m\textit{\textbf{u}}_k,\widetilde{\textit{\textbf{p}}}) E\left( \int\limits_{\left(P_n[\mathbb{0},t\mathbb{1}]\right)^k} \prod\limits_{i=1}^k\left( \Lambda_{m} (\textit{\textbf{u}}_i) \right) \prod\limits_{i=1}^k d\textit{\textbf{u}}_i \left( \int\limits_{P_n[\mathbb{0},t\mathbb{1}]} \Lambda_m(\textit{\textbf{v}}) d\textit{\textbf{v}} \right)^{q/\widetilde{q}} \right),\] where $\widetilde{\textit{\textbf{p}}} = (\widetilde{p},\widetilde{p},...,\widetilde{p}).$

By Assumption {\rm \ref{assump2}}
\[\min_{\substack{\textit{\textbf{u}}_i\in P_n[\mathbb{0},t\mathbb{1}], \\ i=\overline{1,k}}} \rho(b^m\textit{\textbf{u}}_1,...,b^m\textit{\textbf{u}}_k,\widetilde{\textit{\textbf{p}}}) \geq \rho(\mathbb{0}, b^m t\mathbb{1},..., b^m (k-1) t\mathbb{1}, \widetilde{\textit{\textbf{p}}}).\] Thus,
\[ E A^q_{m+1}(t\mathbb{1}) \geq \rho(\mathbb{0}, b^m t\mathbb{1},..., b^m (k-1) t\mathbb{1}, \widetilde{\textit{\textbf{p}}}) E \left( \int\limits_{P_n[\mathbb{0},t\mathbb{1}]} \Lambda_m(\textit{\textbf{u}}) d\textit{\textbf{u}} \right)^{q} \]
\[ = E A^q_{m}(t\mathbb{1}) E\Lambda^q(\mathbb{0}) \frac{\rho(\mathbb{0}, b^m t\mathbb{1},..., b^m (k-1) t\mathbb{1}, \widetilde{\textit{\textbf{p}}})}{E\Lambda^q(\mathbb{0})}\]
\[ \geq E A^q_{m}(t\mathbb{1}) E\Lambda^q(\mathbb{0}) \frac{\rho(\mathbb{0}, b^{m-m_t}  \mathbb{1},..., b^{m-m_t} (k-1)  \mathbb{1}, \widetilde{\textit{\textbf{p}}})}{E\Lambda^q(\mathbb{0})}.\]
By applying the above recursive estimation, one gets
\[ EA_{m_t+1}^q(t\mathbb{1}) \geq C E A^q_{1}(t\mathbb{1}) (E\Lambda^q(\mathbb{0}))^{m_t } \prod_{i=1}^{m_t } \frac{\rho(\mathbb{0}, b^{i-m_t}  \mathbb{1},..., b^{i-m_t} (k-1)  \mathbb{1}, \widetilde{\textit{\textbf{p}}})}{E\Lambda^q(\mathbb{0})} \]
\[ \geq C t^{nq-\log_bE\Lambda^q(\mathbb{0})} \prod_{i=1}^{\infty} \frac{\rho(\mathbb{0}, b^{-i}  \mathbb{1},..., b^{-i} (k-1)  \mathbb{1}, \widetilde{\textit{\textbf{p}}})}{E\Lambda^q(\mathbb{0})}\geq Ct^{nq-\log_bE\Lambda^q(\mathbb{0})}.\] As $A_{m}^q(t\mathbb{1})$ is a submartingale $EA^q(t\mathbb{1})\geq EA_{m_t+1}^q(t\mathbb{1}),$ which provides the required boundary.

For $q\in(0,1)$ the estimate from below can be found by using  H\"{o}lder's inequality \eqref{holder_upper} with $\widetilde{p}=q.$ Then,
\[A_{m+1}^q(t\mathbb{1}) = \left( \int\limits_{P_n[\mathbb{0}, t \mathbb{1}]} \prod_{l=0}^{m_{\textit{\textbf{t}}}-1} \Lambda^{(l)}(b^l\textit{\textbf{u}}) \prod_{l=m_{\textit{\textbf{t}}}}^{m} \Lambda^{(l)}(b^l\textit{\textbf{u}} )d\textit{\textbf{u}} \right)^q \]
\[\geq \left( \int\limits_{P_n[\mathbb{0},t \mathbb{1}]} \left( \prod_{l=0}^{m_{\textit{\textbf{t}}}-1} (\Lambda^{(l)}(b^l\textit{\textbf{u}} ))\right)^q \prod_{l = m_{\textit{\textbf{t}}}}^{m}\Lambda^{(l)}(b^l\textit{\textbf{u}} )  d\textit{\textbf{u}}\right) \left(\int\limits_{P_n[\mathbb{0},t\mathbb{1}]} \prod_{l = m_{\mathfrak{t}}}^{m}\Lambda^{(l)}(b^l\textit{\textbf{u}} )  d\textit{\textbf{u}}\right)^{q/p}.\]
By applying the identity \eqref{moments_estimation_upper}, one obtains
\[ EA_{m+1}^q(t\mathbb{1}) \geq  (E\Lambda^q(\mathbb{0}))^{m_t} \left(b^{-m_t }\right)^{nq} E A_{m-m_t+1}^q(b^{\log_bt + m_{\textit{\textbf{t}}}}\mathbb{1}). \]  As $-1-\log_bt\leq m_t \leq -\log_b t, $ $b^{-m_tnq}\geq t^{nq},$ and $(E\Lambda^q(\mathbb{0}))^{m_t} \geq \frac{t^{-\log_bE\Lambda^q(\mathbb{0})}}{E\Lambda^q(\mathbb{0}))}.$ Therefore,
\[ EA_{m+1}^q(t\mathbb{1}) \geq t^{nq-\log_bE\Lambda^q(\mathbb{0})} \frac{ E A_{m-m_t+1}^q(b^{\log_bt + m_{\textit{\textbf{t}}}}\mathbb{1})}{E\Lambda^q(\mathbb{0})}.\]
As $A_{m-m_t+1}^q$ is a supermartingale,  $EA_{m-m_t+1}^q(b^{\log_bt + m_{\textit{\textbf{t}}}}\mathbb{1}) \geq EA^q(b^{\log_bt + m_{\textit{\textbf{t}}}}\mathbb{1}),$ and
\[ EA_{m+1}^q(t\mathbb{1}) \geq C t^{nq-\log_bE\Lambda^q(\mathbb{0})},\] which completes the proof.\end{proof}

\begin{theorem}
\label{renyi_func_theorem}
Let the conditions of Lemma {\rm \ref{lemma_moments}} be satisfied for all $q\in[0,p], \ p>0.$ Then the limiting measure $\mu(\cdot)$ from Corollary {\rm \ref{measure}} exists and possesses the following R{\'e}nyi function
\[ T(q) = q-1-\frac{1}{n}\log_bE\Lambda^q(\mathbb{0}), \ q\in[0,p].\]
\end{theorem}

\begin{proof}
Using the homogeneity of the limit measure $\mu(\cdot)$ one obtains
\[ T(q) = \liminf\limits_{j\to\infty} \frac{\log_2E\sum_l\mu^q\left(B_l^{(j)}  \right)}{\log_2\left|  B_0^{(j)}  \right| }  =  \liminf\limits_{j\to\infty}\frac{jn+\log_2E\mu^q\left(B_{0}^{(j)}\right)}{\log_2\left|B_{0}^{(j)}\right|}.\]

Now let us estimate the R{\'e}nyi function of $\mu(\cdot)$ from above. The set $B_0^{(j)}$ is the hypercube with Lebesgue measure $|B_0^{(j)}| = 2^{-nj}$.  For sufficiently large $j$ one can apply Lemma \ref{lemma_moments} and obtain that $E\mu^q(B_0^{(j)}) = EA^q(\frac{1}{2^j}\mathbb{1}) \leq C 2^{-j(nq-\log_bE\Lambda^q(\mathbb{0}))}.$ Thus, the following estimate \mbox{holds true}

\[T(q) \leq \liminf\limits_{j\to\infty}\frac{jn+\log_2C - j(nq-\log_bE\Lambda^q(\mathbb{0}))}{-jn} = q - 1 - \frac{1}{n}\log_bE\Lambda^q(\mathbb{0}). \]

The same estimate from below follows from the lower bound in \mbox{Lemma {\rm \ref{lemma_moments}}}. \end{proof}

\section{Sub-Gaussian geometric scenario}
\label{sec4}
The conditions presented in Theorem \ref{th2} are stated in terms of $q$-th moments of the underlying random fields $\Lambda(\cdot).$ The calculation of these high-order moments might involve applications of special methods, see, for example {\rm\cite[Section {\rm 1.3}]{Leon}}. However, for some random fields $\Lambda(\cdot)$ the conditions of Theorem \ref{th2} can be restated in terms of second moments of $\Lambda(\cdot),$ which makes these conditions easier to check. This section shows applications of the obtained general results for the sub-Gaussian geometric case.

First, this section provides the main definitions and notations from the theory of $\varphi$-sub-Gaussian random variables. These definitions and notations are adapted from the monograph \cite{BK}. Then, the second moment conditions that guarantee the  $L_q$ convergence of $A_m(\textit{\textbf{t}}), \ \textit{\textbf{t}}\in P_n[\mathbb{0}, \mathbb{1}],$ as $m\to\infty$, are provided.

The class of sub-Gaussian random variables is a natural extension of the class of Gaussian random variables. Tail distributions of sub-Gaussian random variables behave similarly to the Gaussian ones and their sample path properties rely on their mean square regularity. One of the main classical tools  to study the boundedness
of sub-Gaussian processes is the generic chaining (majorizing measures) method. There is a rich and well-developed theory on  sub-Gaussian random variables and processes, which can be found in Ledoux and Talagrand \cite{ledoux1996isoperimetry, ledoux1991probability} and references therein. The space of $\varphi$-sub-Gaussian random variables was introduced in \cite{kozachenko1985banach} to generalize the class of sub-Gaussian distributions. The sub-Gaussian distributions belong to the $\varphi$-sub-Gaussian class with $\varphi(x)=x^2/2.$ More properties of $\varphi$-sub-Gaussian distributions and processes, their applications in stochastic approximation theory and wavelet representations can be found in \cite{ giuliano2003spaces, antonini2008convergence, kozachenko2016whittaker, kozachenko2013convergence, kozachenko2015convergence}.
\begin{definition}\label{n_function} A continuous function $\varphi(x), x\in {\mathbb{R}}$, is called an Orlicz $N$-{function} if

a) it is even and convex,

b) $\varphi(0)=0$,

c) $\varphi(x)$  is a monotone increasing function for $x>0$,

d) $\lim\limits_{x\to 0} \varphi(x)/x =0$ and $\lim\limits_{x\to+\infty}  \varphi(x)/x =+\infty$.

\end{definition}

\begin{example}
{\rm The function $\varphi(x)= |x|^r/r, \ r> 1,$ is an Orlicz $N$-function.}
\end{example}

In the following the notation $\varphi(x)$ stands for an Orlicz $N$-function.

\begin{definition} A function $\psi(x):=\sup_{y\in {\mathbb{R}}} \left(xy - \varphi(y) \right), \ x\in {\mathbb{R}},$ is called the Young-Fenchel transform of $\varphi(x)$.
\end{definition}

\begin{example}
{\rm The Young-Fenchel transform of $\varphi(x)= |x|^p/p, \ p>1$, is $\psi(x)= |x|^r/r,$ where $1/p + 1/r =1$.}
\end{example}
Any Orlicz $N$-function $\varphi(x)$ can be represented in the integral form
\begin{equation}\label{density_N_function}
\varphi(x)= \int_0^{|x|} p_{\varphi}(t)\ dt,
\end{equation} where $p_{\varphi}(t)$, $t\geq0,$ is its density. The density $p_{\varphi}(\cdot)$ is nondecreasing and, as a consequence, the function $\varphi(\cdot)$ is increasing, differentiable, and $\varphi'(\cdot)=p_{\varphi}(\cdot)$.

\begin{definition}{\rm \cite{BK}}\label{def4} A zero mean random variable $X$ is $\varphi$-sub-Gaussian if there exists a finite positive constant $a$ such that  $E\exp\left(t X\right) \le \exp\left(\varphi(at)\right)$ for
all $t\in {\bf \mathbb{R}}$.
The $\varphi$-sub-Gaussian norm $\tau_{\varphi}(X)$ is defined as
\[ \tau_{\varphi}(X):=\inf\{a>0: E\exp\left(t X\right) \le \exp\left(\varphi(at)\right), \ t\in {\mathbb{R}} \}. \]
\end{definition}

The space $Sub_\varphi(\Omega)$ of all $\varphi$-sub-Gaussian random variables is a Banach space with respect to the norm $\tau_\varphi$.

$\varphi$-sub-Gaussianity allows to estimate tails of distributions. The following result holds.

\begin{lemma}{\rm \cite[Lemma 4.3, p.~66]{BK}}
\label{l1}
If $\varphi(\cdot)$ is an Orlicz $N$-function and a random variable $X\in Sub_\varphi(\Omega)$, then for all $x> 0$ it holds
\[P\big(X \geq x\big) \le \exp\left(-\psi\left(\frac{x}{\tau_{\varphi}(X)}\right)\right).\]
\end{lemma}

\begin{definition}{\rm\cite{kozachenko2016whittaker}}\label{subgaussian_set}
A set $\Delta$ of random variables from $Sub_{\varphi}(\Omega)$ is called strictly $\varphi$-sub-Gaussian if there exists a constant $D_\Delta > 0$ such that for all finite sets $I,$ $\lambda_i\in \mathbb{R},$ and $\xi_i \in \Delta, \ i\in I, $ it holds
\[ \tau_{\varphi}\left(\sum_{i\in I}\lambda_i \xi_i\right) \leq D_\Delta \left( E \left(\sum_{i\in I}\lambda_i \xi_i\right)^2  \right)^{1/2}. \]
The constant $D_\Delta$ is called a defining constant.
\end{definition}

\begin{definition}{\rm\cite{kozachenko2016whittaker}}
A random field $X(\textit{\textbf{s}}), \ \textit{\textbf{s}}\in \mathbb{R}^n, \ n \geq 1,$ is called strictly $\varphi$-sub-Gaussian if $\sup_{s\in\mathbb{R}^n}\tau_\varphi(X(\textit{\textbf{s}})) < \infty$ and the set of random variables $\{ X(\textit{\textbf{s}}), \ \textit{\textbf{s}}\in\mathbb{R}^n \}$ is strictly $\varphi$-sub-Gaussian. The defining constant of this set is called the defining constant of the random field $X(\cdot)$ and is denoted by $D_X.$
\end{definition}

\begin{assumption}
\label{assump_sub}
Let $\Lambda(\textit{\textbf{s}}) = e^{X(\textit{\textbf{s}})}/Ee^{X(\mathbb{0})},\ \textit{\textbf{s}} \in \mathbb{R}^n, \ n \geq 1,$ where $X(\cdot)$ is a homogeneous, isotropic strictly $\varphi$-sub-Gaussian random field with the defining constant $D_X$ such that its covariance function $\rho_X(||\textit{\textbf{u}}||) = E(X(\mathbb{0}) X(\textit{\textbf{u}})),\ \textit{\textbf{u}} \in\mathbb{R}^n,$ is nonincreasing in $||\textit{\textbf{u}}||.$
\end{assumption}

In what follows, analogously to previous sections,  $\Lambda^{(i)}(\textit{\textbf{s}})=e^{X^{(i)}(\textit{\textbf{s}})}/Ee^{X^{(i)}(\mathbb{0})},$ $\textit{\textbf{s}} \in \mathbb{R}^n, \ i=0,1,...,$ where $ X^{(i)}(\cdot), \ i=0,1,...,$ is an infinite collection of independent stochastic copies of $X(\cdot)$.

The following example shows that wide classes of $\varphi$-sub-Gaussian random fields with a given determining constant can be easily constructed.

\begin{example}\label{example_sub} Note that to construct random variables $A_m(\textit{\textbf{t}}),\ m \geq1,$ it is sufficient to define random fields $\Lambda^{(i)}(\cdot), \ i=\overline{0,m-1}.$

Let $\varphi(\sqrt{x})$ be concave and $\{ \xi_j, \ j=1,2,... \} $ be a family of independent $Sub_{\varphi}(\Omega)$ random variables such that there exists such $D>0$ that $\tau_\varphi(\xi_j) \leq D(E\xi_j^2)^{1/2}$ for any $j= 1,2....$ Consider a sequence of nonrandom functions $\{ f_j(\textit{\textbf{s}}), \ j\geq1 \}$, such that the series $\sum_{j=1}^{\infty}f_j(\textit{\textbf{s}})$ converges for $\textbf{\textit{s}}\in P_n[\mathbb{0}, \mathbb{1}],$ then
\[
X (\textit{\textbf{s}}) = \sum\limits_{j=1}^\infty \xi_j f_j (\textit{\textbf{s}} ), \ \textit{\textbf{s}} \in P_n[\mathbb{0},  \mathbb{1}],\]
 is a strictly $\varphi$-sub-Gaussian random field with the determining constant $D,$ see {\rm \cite[Example 3.10]{kozachenko2018estimates}.}
\end{example}

In what follows, for the simplicity of the presentation, we denote $\widetilde{\varphi}(\cdot) : = \varphi(\sqrt{\cdot}).$ Note that the function $\widetilde{\varphi}(\cdot)$ is finite on each bounded interval as $\varphi(\cdot)$ is continuous.

\begin{theorem}
\label{th_sub}
Let Assumption {\rm \ref{assump_sub}} be satisfied. Suppose that
\begin{equation}\label{varphi_bound}Ee^{X(\mathbb{0})}>1 \ {\rm or} \ \frac{p_{\varphi}(x)}{|x|}<C<+\infty
\end{equation} in some neighbourhood of $0.$ If for $\textit{\textbf{p}} = (p_1,p_2,...,p_k),\ p_j\geq1, \ j=\overline{1,k}, \ k\geq2,$ such that $\sum_{j=1}^kp_j = p,$ it holds
\begin{equation}
\label{cond1_sub}
 b > \exp\left({\frac{1}{n}}\left(\widetilde{\varphi}\left(p^2D^2_XEX^2(\mathbb{0})\right)- p\ln\left( E e^{X(\mathbb{0})} \right)\right)\right),
\end{equation}
\begin{equation}
\label{cond2_sub}
\sum_{i=0}^\infty\left(D_X^2\sum\limits_{l,h=1}^kp_lp_h\rho_X(||b^i(l-h)\mathbb{1}||)-\widetilde{\varphi}^{-1}\left(p\ln\left( E e^{X(\mathbb{0})}\right)\right)\right)<+\infty,
\end{equation} then for every fixed $\textit{\textbf{t}} \in P_n[\mathbb{0},\mathbb{1}]$ and for all $q\in[0,p],$ the random variables $A_m(\textit{\textbf{t}})$  converge to some random variables $A(\textit{\textbf{t}})$ in the spaces $L_q,$ as $m \to \infty$.

\end{theorem}

\begin{proof}
First, let us show that Assumption \ref{assump2'} is satisfied with the following function
\begin{equation}\label{Rho}
\widetilde{\rho}(\textit{\textbf{u}}_1,\textit{\textbf{u}}_2,...,\textit{\textbf{u}}_k, \textit{\textbf{p}}) :=  \exp\bigg(\widetilde{\varphi}\bigg(D^2_X\sum\limits_{\substack{l,h=1}}^kp_i p_j\rho_X(||\textit{\textbf{u}}_l-\textit{\textbf{u}}_h||)\bigg)\bigg)- p\ln\left( E e^{X(\mathbb{0})}\right)\bigg).
\end{equation}

The random field $X(\cdot)$ is strictly $\varphi$-sub-Gaussian. Thus, by Definitions \ref{def4} and \ref{subgaussian_set} the following estimate \mbox{holds true}
\[E\exp\bigg(\sum_{l=1}^kp_lX(\textit{\textbf{u}}_l)\bigg) \leq \exp\bigg(\varphi\bigg(\tau_\varphi\bigg(\sum\limits_{l=1}^kp_lX(\textit{\textbf{u}}_l)\bigg)\bigg)\bigg) \]
\[ \leq \exp\bigg(\varphi\bigg(D_X \bigg( E \bigg(\sum\limits_{l=1}^kp_lX(\textit{\textbf{u}}_l)\bigg)^2  \bigg)^{1/2}\bigg)\bigg)=  \exp\bigg(\widetilde{\varphi}\bigg(D^2_X \sum\limits_{\substack{l,h=1 }}^kp_l p_h\rho_X(||\textit{\textbf{u}}_l-\textit{\textbf{u}}_h||)\bigg)\bigg)\bigg).\]
By Assumption \ref{assump_sub}, one obtains
\[ E\left( \prod_{l=1}^k\Lambda^{p_l}(\textit{\textbf{u}}_l)\right) \leq \exp\bigg(\widetilde{\varphi}\bigg(D^2_X\sum\limits_{\substack{l,h=1 }}^kp_l p_h\rho_X(||\textit{\textbf{u}}_l-\textit{\textbf{u}}_h||)\bigg)-p\ln\left(E e^{X(\mathbb{0})}\right)\bigg).\] As $\rho_X(\cdot)$ is a nonincreasing function, Assumption \ref{assump2'} is satisfied.

Note that by Jensen's inequality $Ee^{X(\mathbb{0})}\geq e^{EX(\mathbb{0})}=1.$ Thus, $\ln\left(E e^{X(\mathbb{0})} \right) \geq 0.$

We will prove the convergence of $A_m(\textit{\textbf{t}})$ in the space $L_p,\ p\geq 2,$ by checking the conditions of Corollary \ref{cor_for_sub}.

By the definition of function $\widetilde{\rho}(\cdot),$ one can see that
\[
\widetilde{\rho}^{\frac{1}{n}}(\mathbb{0}, \mathbb{0},..., \mathbb{0}, \textit{\textbf{p}}) = \exp\left({\frac{1}{n}}\left(\widetilde{\varphi}\left(p^2D^2_XEX^2(\mathbb{0})\right)-p\ln\left(E e^{X(\mathbb{0})}\right)\right)\right).
\] Note that $\widetilde{\rho}^{\frac{1}{n}}(\mathbb{0}, \mathbb{0},..., \mathbb{0}, \textit{\textbf{p}})$ is finite and the condition \eqref{cond2.1'} is satisfied if
\[  b > \exp\left({\frac{1}{n}}\left(\widetilde{\varphi}\left(p^2 D^2_XEX^2(\mathbb{0})\right)-p\ln\left(E e^{X(\mathbb{0})}\right)\right)\right),\]
which is the condition \eqref{cond1_sub} of the theorem.

The next step is to check the condition \eqref{cond2.2'}. Lets consider
\[ \widetilde{\rho}(\mathbb{0}, b^i\mathbb{1}, 2b^i\mathbb{1},...,(k-1)b^i\mathbb{1},\textit{\textbf{p}})\]
\[ = \exp\left(\widetilde{\varphi}\left(D^2_X \sum\limits_{\substack{l,h=1}}^kp_l p_h\rho_X(||b^i(l-h)\mathbb{1}||)\right) - p\ln\left(E e^{X(\mathbb{0})}\right) \right).\] By the mean value theorem there exists such
\begin{equation}\label{interval}\eta\in \left[\widetilde{\varphi}^{-1}\left(p\ln\left(E e^{X(\mathbb{0})} \right) \right), \ \left(D^2_X\sum\limits_{\substack{l,h=1}}^kp_l p_h \rho_X(||b^i(l-h)\mathbb{1}||)\right)\right]\end{equation} that
\[\widetilde{\varphi}\left(D^2_X\sum\limits_{\substack{l,h=1}}^kp_l p_h \rho_X(||b^i(l-h)\mathbb{1}||)\right) - \widetilde{\varphi}\left(\widetilde{\varphi}^{-1}\left(p\ln\left(E e^{X(\mathbb{0})} \right) \right) \right) \]
\[ = \widetilde{\varphi}^{'}(\eta) \left(D^2_X\sum\limits_{\substack{l,h=1}}^kp_l p_h \rho_X(||b^i(l-h)\mathbb{1}||) - \widetilde{\varphi}^{-1}\left(p\ln\left(E e^{X(\mathbb{0})} \right) \right) \right).\] According to \eqref{density_N_function}, $\widetilde{\varphi}'(\eta) = p_{\varphi}(\sqrt{\eta})/(2\sqrt{\eta}),$ where $p_{\varphi}(\cdot)$ is the density of $\varphi(\cdot).$

Let \[C_1:=\sup_{x}p_{\varphi}(\sqrt{x})/(2\sqrt{x}),\] where $x$ is from the interval \eqref{interval}. Then, by \eqref{varphi_bound} $C_1<+\infty,$ and
\[ \widetilde{\rho}(\mathbb{0}, b^i\mathbb{1}, 2b^i\mathbb{1},...,(k-1)b^i\mathbb{1}, \textit{\textbf{p}}) \] \[ \leq \exp\left(C_1 \left( D^2_X\sum\limits_{\substack{l,h=1}}^kp_l p_h \rho_X(||b^i(l-h)\mathbb{1}||) - \widetilde{\varphi}^{-1}\left(p\ln\left(E e^{X(\mathbb{0})} \right) \right)\right)\right),\] and condition \eqref{cond2.2'} is satisfied when \eqref{cond2_sub} holds true. Thus, by Corollary \ref{cor_for_sub} the random variables $A_m(\textbf{\textit{t}})$ converge to $A(\textbf{\textit{t}})$ in $L_q$.
\end{proof}

\begin{remark}
As $\rho_X(||b^i(l-h)\mathbb{1}||)\leq\rho_X(||b^i\mathbb{1}||)=\rho_X(\sqrt{n}b^i),\ i\geq0,\ l\neq h,$ the series in {\rm \eqref{cond2_sub}}  can be estimated as
\[\sum_{i=0}^\infty\left(D_X^2\sum\limits_{l,h=1}^kp_lp_h\rho_X(||b^i(l-h)\mathbb{1}||)-\widetilde{\varphi}^{-1}\left(p\ln\left( E e^{X(\mathbb{0})}\right)\right)\right)\]
\[ \leq \sum_{i=0}^\infty\left(D_X^2\left(\sum\limits_{\substack{l,h=1 \\ l \neq h}}^kp_lp_h\rho_X(||\sqrt{n}b^i||) + \sum\limits_{\substack{l=1}}^kp^2_l\rho_X(0) \right)-\widetilde{\varphi}^{-1}\left(p\ln\left( E e^{X(\mathbb{0})}\right)\right)\right).\]
\end{remark}

\begin{remark}\label{remark_sub}
Let $p_1=p_2=...=p_k.$ Then,
\[ \sum_{l,h=1}^k p_l p_h \rho_X(||b^i(l-h)\mathbb{1}||) = p_1^2\left( k\rho_X(\mathbb{0}) + 2\sum_{j=1}^{k-1}(k-j)\rho_X(\sqrt{n}b^ij) \right).\]

Indeed, the first double sum consists of $k$ "diagonal" elements (where $l=k$), each equals to $p^2_1\rho_X(0)$ and $2(k-j)$ "off-diagonal" elements (where $|l-h|=j),$ each equals to $p^2_1\rho_X(||b^i j \mathbb{1}||).$

As $\rho_X(\cdot)$ is a nonincreasing function, the second sum shows that the summability condition is required only for $j=1,$ i.e. one can use $\sum_{i=0}^{\infty}\rho_X(\sqrt{n}b^i)<+\infty.$ \end{remark}

The next corollary gives specific sufficient conditions on a wide class of Orlicz $N$-functions $\varphi(\cdot)$ that guarantee that Theorem \eqref{th_sub} holds for the corresponding $\varphi$-sub-Gaussian fields.

\begin{corollary} Let Assumption {\rm \ref{assump_sub}}, and \eqref{cond1_sub} of Theorem {\rm \ref{th_sub}} be satisfied, $Ee^{X(\mathbb{0})}>1,$ and there exist $C_1>0, \ x_0>0,$ such that for all $x\geq x_0$
\begin{equation}\label{cor_sub_cond_1}
\varphi(x)\leq C_1x^{\beta}, \ \beta\in(1,2).
\end{equation}
\begin{enumerate}[label=(\roman*)]
\item If
\begin{equation}\label{cor_sub_cond_2}
\sum_{i=0}^{\infty} \rho_X(\sqrt{n}b^i)<+\infty,
\end{equation} then for $p_j\leq C_0, \ j=\overline{1,k},$ and sufficiently large $k$ the condition \eqref{cond2_sub} and the statement of Theorem {\rm \ref{th_sub}} hold true.

\item If $\Lambda(\cdot)$ is $p$-weakly associated and there exist $C_2,$ $x_0$ such that for all $x\geq x_0$
\begin{equation*}
0\leq \rho_X(\sqrt{n}x) \leq C_2 x^{-\alpha}, \ \alpha>n,
\end{equation*}
and $b>e^{a(p)}, $  where $p$ is an even integer and $a(p) = \frac{1}{n}\left(\widetilde{\varphi}\left(p^2D^2_X\rho_X(0)\right)-p\ln Ee^{X(\mathbb{0})}\right),$ then for all $\textit{\textbf{t}}\in P_n[\mathbb{0},\mathbb{1}]$ and $q \in [2,p]$
\begin{equation*}
E|A(\textit{\textbf{t}})-A_m(\textit{\textbf{t}})|^q\leq C \bigg(\prod\limits_{i=1}^nt_i\bigg)^{q-1}\left(\frac{e^{na(p)}}{b^n}\right)^m.
\end{equation*}
\end{enumerate}
\end{corollary}

\begin{proof}
\textit{(i)} It follows from \eqref{cor_sub_cond_2} that the condition \eqref{cond2_sub} is satisfied if
\[ D^2_X \sum_{l=1}^k p^2_l \rho_X(\mathbb{0}) \leq \widetilde{\varphi}^{-1}\left(p\ln\left( Ee^{X(\mathbb{0})} \right)\right).\] The application of $\widetilde{\varphi}(\cdot)$ to the both sides of this inequality and \eqref{cor_sub_cond_1} give
\begin{equation}\label{++} \widetilde{\varphi}\left( D^2_X \sum_{l=1}^k p^2_l \rho_X(\mathbb{0}) \right) \leq C_1D_X^{\beta} \left(\sum_{l=1}^k p^2_l\right)^{\beta/2} \rho^{\beta/2}_X(\mathbb{0}) \leq \sum_{l=1}^kp_l\ln\left( Ee^{X(\mathbb{0})} \right).\end{equation} The last inequality holds true if

\[ \frac{\left( \sum_{l=1}^k p^2_l \right)^{\beta/2}}{\sum_{l=1}^kp_l} \leq \frac{\ln\left( Ee^{X(\mathbb{0})} \right)}{C_1 D_X^{\beta} \rho^{\beta/2}(0)}.\] Noting that for $\beta<2$

\[ \frac{\left( \sum_{l=1}^k p^2_l \right)^{\beta/2}}{\sum_{l=1}^kp_l} \leq \frac{k^{\beta/2}\max_{l=\overline{1,k}}p^2_l}{k\min_{l=\overline{1,k}}p_l}\to 0, \ k\to\infty,\] one obtains \eqref{++} for sufficiently large $k.$

\textit{(ii)} From \eqref{Rho} it follows that for the vector $\textit{\textbf{p}} = (1,1,...,1)$ that consists of $p$ components
\[ \widetilde{\rho}(\mathbb{0},x\mathbb{1},2x\mathbb{1},...,(p-1)x\mathbb{1}) \] \[ = \exp\left(\widetilde{\varphi}\left(D^2_XpEX^2(\mathbb{0})+D^2_Xp(p-1)\rho_X(\sqrt{n}x)\right)-\widetilde{\varphi}\left(\widetilde{\varphi}^{-1}(p\ln(Ee^{X(\mathbb{0})}))\right)\right).\] As $\Lambda(\cdot)$ is $p$-weakly associated, Remark \ref{remark_assoc_main} implies that $\widetilde{\rho}(\cdot)\geq 1.$

From Definition \ref{n_function} one gets that  $\widetilde{\varphi}(x)$ is an increasing function for sufficiently large $x$. By choosing large $p,$ from the proof of \textit{(i)} it follows that $D^2_X   p \rho_X(\mathbb{0}) \leq \widetilde{\varphi}^{-1}\left(p\ln\left( Ee^{X(\mathbb{0})}\right)\right).$ Thus, $\widetilde{\varphi}\left(D^2_X  p \rho_X(\mathbb{0})\right) \leq \widetilde{\varphi}\left(\widetilde{\varphi}^{-1}\left(p\ln\left( Ee^{X(\mathbb{0})}\right)\right)\right) =  p\ln\left( Ee^{X(\mathbb{0})}\right),$ and
\[ \widetilde{\rho}(\mathbb{0},x\mathbb{1},2x\mathbb{1},...,(p-1)x\mathbb{1}) \leq \exp\left(\widetilde{\varphi}(D^2_Xp\rho_X(0)+D^2_Xp(p-1)\rho_X(\sqrt{n}x))-\widetilde{\varphi}\left(D^2_X p \rho_X(\mathbb{0})\right)\right). \] It follows from the mean value theorem that there exists such a constant $C_3:=\sup_x\widetilde{\varphi}'(x),$ $x\in[D^2_Xp\rho_X(0), D^2_Xp^2\rho_X(0)],$ that
\[\exp\left(\widetilde{\varphi}(D^2_Xp\rho_X(0)+D^2_Xp(p-1)\rho_X(\sqrt{n}x))-\widetilde{\varphi}(D^2_Xp\rho_X(0))\right) \leq e^{C_3D^2_Xp(p-1)\rho_X(\sqrt{n}x)}.\] Thus, $\ln \widetilde{\rho}(\mathbb{0},x\mathbb{1},2x\mathbb{1},...,(p-1)x\mathbb{1}) \leq C_2 x^\alpha,\ x\geq x_0,$ if $\rho(\sqrt{n}x)\leq C_2x^\alpha,\ x\geq x_0.$

Noting that $E\Lambda^p(\mathbb{0}) \leq \exp\left({\widetilde{\varphi}(p^2D^2_X\rho_X(0))-p\ln Ee^{X(\mathbb{0})}}\right)$ and applying Corollary \ref{cor_widetilde_rho} complete the proof. \end{proof}

Finally, we present some results for the geometric Gaussian case, which is a specific example with $\varphi(x)=x^2/2.$

\begin{corollary}
\label{gauss}
Let Assumption {\rm \ref{assump_sub}} be satisfied for $\textit{\textbf{p}} = (1,1,...,1),$ $p \geq 2,$  and $X(\cdot)$ is a zero-mean homogeneous and isotropic Gaussian random field. If $b > \exp\left({\frac{p(p-1)EX^2(\mathbb{0})}{2n}}\right)$ and \eqref{cor_sub_cond_2} holds true, then

\begin{enumerate}[label=(\roman*)]
\item For every fixed $\textit{\textbf{t}} \in P_n[\mathbb{0},\mathbb{1}]$ and for all $q\in[0,p],$ the random variables $A_m(\textit{\textbf{t}})$  converge to some random variables $A(\textit{\textbf{t}})$ in the spaces $L_q,$ as $m \to \infty;$
\item If $p$ is an even integer and there exist such $x_0$ and $C_1>0$ that for $x\geq x_0$
\begin{equation*}
0\leq\rho_X(\sqrt{n}x)\leq C_1x^{-\alpha}, \ \alpha> \frac{p(p-1)EX^2(\mathbb{0})}{2 \ln b},
\end{equation*} then for all $\textit{\textbf{t}}\in P_n[\mathbb{0},\mathbb{1}]$ and $q\in[2,p]$
\begin{equation*}
E|A(\textit{\textbf{t}})-A_m(\textit{\textbf{t}})|^q\leq C \left(\prod_{i=1}^n t_i\right)^{q-1} \left(\frac{\exp \left( \frac{p(p-1)EX^2(\mathbb{0})}{2}\right)}{b^{n }}\right)^{m}.
\end{equation*}
If
$b > \exp\left({\frac{\gamma p(p-1)EX^2(\mathbb{0})}{2n(\gamma-1)}} \right)$
 for some $\gamma\in(\max(1,n/\alpha), n\ln (b)/ln(E\Lambda^p(\mathbb{0}))),$ then for all $\textit{\textbf{t}}\in P_n[\mathbb{0},\mathbb{1}]$ and $q\in[2,p]$
\begin{equation*}
E|A(\textit{\textbf{t}})-A_m(\textit{\textbf{t}})|^q\leq C  \left( \frac{ \prod_{i=1}^n t_i ^{q-1}}{b^{nm}}\right)^{1/\gamma}.
\end{equation*}
\item It holds
\[ \lim\limits_{m\to\infty} \mu_m(B_j) = \mu(B_j) \ \ \ \ a.s.,\] where $\mathfrak{B} = \{B_j: \ B_j\subseteq P_n[\mathbb{0} ,\mathbb{1}]\}$ is a finite or countable family of Borel sets, and if
\[ \sum_{m=1}^{\infty} (\rho_X(0)-\rho_X(b^{-m}\mathbb{1})) < \infty,\] then the random measure $\mu(\cdot)$ possesses the following  R{\'e}nyi function
\[ T(q) = q-1-\frac{1}{n}\log_be^\frac{q(q-1)EX^2(\mathbb{0})}{2}, \ q\in[0,p].\]

\end{enumerate}

\end{corollary}

\begin{proof}
\textit{(i)} Let us check the conditions of Theorem \ref{th_sub} when $\textit{\textbf{p}}=(1,1,...,1)$ consists of $p$ elements equal $1$. As $X(\cdot)$ is a Gaussian random field, it belongs to the class of $\varphi$-sub-Gaussian random fields with $\widetilde{\varphi}(x)=x/2,$ $D_X=1$, and one gets
\[  \exp\left({\frac{1}{n}}\left(\widetilde{\varphi}\left(p^2D^2_XEX^2(\mathbb{0})\right)- p\ln\left( E e^{X(\mathbb{0})} \right)\right)\right) =   e^{\frac{p(p-1)EX^2(\mathbb{0})}{2n}}.\]
By Remark \ref{remark_sub}
\[ \sum_{i=0}^\infty\left(D_X^2\sum\limits_{l,h=1}^pp_lp_h\rho_X(||b^i(l-h)\mathbb{1}||)-\widetilde{\varphi}^{-1}\left(p\ln\left( E e^{X(\mathbb{0})}\right)\right)\right) \]
\[ = \sum_{i=0}^\infty\left( \sum_{j=1}^{p-1}2(p-j)\rho_X(\sqrt{n}b^i j) + p\rho_X(0) -2p\ln e^{\rho_X(0)/2} \right)  \leq C \sum_{i=0}^\infty \rho_X(\sqrt{n} b^i). \] Hence, the conditions of Theorem \ref{th_sub} are satisfied which implies \textit{(i)}.

\textit{(ii)} By \eqref{moment},
\[ \ln(\rho(\mathbb{0},x\mathbb{1},...,(p-1)x\mathbb{1} )) = \ln\left( \frac{Ee^{X(\mathbb{0})+X(x\mathbb{1})+...+X(x(p-1)\mathbb{1})}}{\left(Ee^{X(\mathbb{0})}\right)^p} \right)\]
\[\leq \ln e^{\frac{p(p-1)}{2}\rho_X(\sqrt{n}x)} = \frac{p(p-1)}{2}\rho_X(\sqrt{n}x),\] and
\[ \frac{\ln(E\Lambda^p(\mathbb{0}))}{\ln b} =  \frac{p(p-1)EX^2(\mathbb{0})}{2 \ln b}.\]

Thus, under the assumption in \textit{(ii)}, the conditions of Corollary \ref{rate_theorem_q} and Theorem \ref{rate_theorem_lq_fast} are satisfied.

\textit{(iii)} Now let us check the conditions of Lemma \ref{lemma_moments}. By \eqref{moment} and Assumption \ref{assump_sub}, for  $\textit{\textbf{q}}=(q-1,1),\ q\in(0,1),$
\[\sum_{i=1}^{\infty} \ln\left(\frac{\rho(\mathbb{0}, b^{-i}\mathbb{1},\textit{\textbf{q}})}{E\Lambda^q(\mathbb{0})}\right) =(1-q)\sum_{i=1}^{\infty}\left(EX^2(\mathbb{0})-\rho_X(\sqrt{n}b^{-i})\right).\] For $q\geq1$ let us choose $\widetilde{\textit{\textbf{p}}}=(\widetilde{p},\widetilde{p}),$ $\widetilde{p}=q/2.$
Then
\[ \sum_{i=1}^{\infty} \ln\left(\frac{\rho(\mathbb{0},b^{-i}\mathbb{1},\widetilde{\textit{\textbf{p}}}))}{E\Lambda^q(\mathbb{0})}\right) = -\widetilde{p}^2 \sum_{i=1}^{\infty}\left(EX^2(\mathbb{0})-\rho_X(\sqrt{n}b^{-i})\right).\] Thus, the conditions of Theorem \ref{renyi_func_theorem} hold if $\sum_{i=1}^{\infty}\left(EX^2(\mathbb{0})-\rho_X(\sqrt{n}b^{-i})\right)<\infty,$ which completes the proof. \end{proof}

\section*{Acknowledgments}

This research was partially supported under the Australian Research Council's Discovery Projects funding scheme (project number  DP220101680). We also would like to thank Professors Antoine Ayache and Nikolai Leonenko for insightful discussions of related topics for random fields.

\bibliographystyle{amsplain}
\bibliography{mybib}
\end{document}